\documentclass[twoside,12pt]{article}
\usepackage{amssymb,amsmath,bm,euscript}
\usepackage{graphicx,color,caption}

\usepackage[colorlinks=true]{hyperref}

\usepackage{algorithm,algpseudocode}
\usepackage{braket}

\newtheorem{proposition}{Proposition}[section]
\newtheorem{theorem}[proposition]{Theorem}
\newtheorem{lemma}[proposition]{Lemma}

\newtheorem{definition}[proposition]{Definition}
\newtheorem{remark}[proposition]{Remark}

\newenvironment{proof}[1][Proof]{\noindent\textbf{#1.} }{\ \rule{0.6em}{0.6em}}

\def\0{{\bf 0}}

\title{\bf{A Parallelizable Method for Missing Internet Traffic Tensor Data}}
\author{ \hspace{1mm} Chen Ling \thanks{Department of Mathematics, School of Science, Hangzhou Dianzi University, Hangzhou, 310018 China; ({\tt macling@hdu.edu.cn}). },
    \ \
Gaohang Yu\thanks{Department of Mathematics, School of Science, Hangzhou Dianzi University, Hangzhou, 310018 China; ({\tt maghyu@hdu.edu.cn}). },
\ \
Liqun Qi \thanks{Department of Applied Mathematics, The Hong Kong Polytechnic University, Hung Hom, Kowloon, Hong Kong; ({\tt liqun.qi@polyu.edu.hk}).}
\ and \
    Yanwei Xu\thanks{Future Network Theory Lab, 2012 Labs
Huawei Tech. Investment Co., Ltd, Shatin, New Territory, Hong Kong, China; ({\tt xuyanwei1@huawei.com}).}}

\begin{document}
\date{\today}
\maketitle
\begin{abstract}
Recovery of internet network traffic data from incomplete observed data is an important issue in internet network engineering and management. In this paper, by fully combining the temporal stability and periodicity features in internet traffic data, a new separable optimization model for internet data recovery is proposed, which is based upon the t-product and the rapid discrete Fourier transform of tensors. Moreover, by using generalized inverse matrices, an easy-to-operate and effective algorithm is proposed.
In theory, we prove that under suitable conditions, every accumulation point of the sequence generated by the proposed algorithm is a stationary point of the established model. Numerical simulation results carried on the widely used real-world internet network datasets, show good performance of the proposed method. In the case of moderate sampling rates, the proposed method works very well, its effect is better than that of some existing internet traffic data recovery methods in the literature. The separable structural features presented in the optimization model provide the possibility to design more efficient parallel algorithms.

\end{abstract}


\section{Introduction}\label{Introd}
Internet network traffic data can often  be arranged in the form of multidimensional arrays. For example, a traffic matrix is often applied to track the volume of traffic between origin-destination (OD) pairs in a network \cite{V96}. Estimating the end-to-end traffic data in a network is an essential part of many network design and traffic engineering tasks, including capacity planning \cite{Cah98}, load balancing \cite{TR13}, network provisioning \cite{RTZ03}, path setup and  anomaly detection \cite{LPCDKT03}, and failure recovery \cite{TR13}. Unfortunately, direct and precise end-to-end flow traffic measurement is very difficult or even infeasible in the traditional IP network. Missing data is unavoidable. Since many traffic engineering tasks require the complete traffic volume information or are highly sensitive to the missing data, the accurate reconstruction of missing values from partial traffic measurements becomes a key problem.

To infer the missing data in the internet network, many research works have been developed, for instance, see \cite{ACRUW06,NJG13} and the references therein. Using optimization technology to recover incomplete internet traffic data is a very important way, which are somewhat different from the existing methods for dealing with general data recovery \cite{CCS10,CR09,CSZZC19,SLH19,YHHH16},  including Compressive Sensing (CS) \cite{Do06},  Singular Value Thresholding (SVT) algorithm \cite{CCS10} and Low-rank Matrix Fitting (LMaFit) algorithm \cite{WYZ12}, etc.
Since most of the known approaches for missing internet network data are designed based on purely spatial or purely temporal information \cite{LPCDKT03,V96,ZRLD05}, sometime their data recovery performance is low. To capture more spatial-temporal information in the traffic data, various modified matrix based models and corresponding optimization algorithms were presented to recover missing traffic data, for example, see \cite{AMDJ16,GC12,RZWQ12,TR15}. The first spatio-temporal optimization model of traffic matrices (TMs) was established in \cite{RZWQ12}, which is designed based on low-rank approximation combined with the spatio-temporal operation and local interpolation. The proposed optimization model in \cite{RZWQ12}  can be formulated as
\begin{equation}\label{MMC}
\begin{array}{ll}
{\rm min}_{L,R}&\|\mathcal{A}(LR^\top)-B\|_F^2+\lambda(\|L\|_F^2+\|R\|_F^2)\\
&+\|S(LR^\top)\|_F^2+\|(LR^\top)T^\top\|_F^2,
\end{array}
\end{equation}
where $\mathcal{A}$ is a linear operator, the matrix $B$ contains the measurements, and $S$ and $T$ are the spatial and temporal constraint matrices, respectively, which express the known knowledge about the spatio-temporal structure of the traffic matrix (e.g., temporally nearby its elements have similar values). Based on the model (\ref{MMC}), the authors applied an alternating least squares  procedure to solve it. Numerical experiments show that the proposed method in \cite{RZWQ12}  has better performance.
Since then, several other matrix recovery optimization models and algorithms \cite{CQZXH14,DCYG13,GC12,LDGL13,MG13,X15,ZLNLZ17} have been proposed to recover the missing data from partial traffic or network latency measurements. Although these approaches enjoy good performance when the data missing ratio is low, their performance suffers when the missing ratio is large, especially in the extreme case when the traffic data on several time intervals are all lost \cite{XPWXWCZQ18}. In fact, such matrix based methods may 
destroy the intrinsic tensor structure of high-dimensional data and also increases the computational cost of data recovery.

 In order to improve the recovery performance of the matrix-based methods mentioned above, several tensor optimization methods have been applied to recover missing traffic data, for example, \cite{ADKM11,DH15,GRY11}. The core of the tensor methods lies in the tensor decomposition, which commonly takes two forms: CANDECOMP/PARAFAC (CP) decomposition \cite{CC70,H70} and Tucker decomposition \cite{Tucker66}. For a network with a location set $\Sigma$, let cardinality $|\Sigma|=N$. As a straightforward way of modeling \cite{ZZXC15}, traffic tensor may be formed with a third order tensor $\mathcal{G}\in \mathbb{R}^{N\times N\times T}$, corresponding respectively to the origin, destination and the total number of time intervals to consider. 
 Traffic data are typically measured over some time intervals, and the value reported is an average. Therefore, the element $g_{ijk}$ in $\mathcal{G}$ is used to represent the traffic from origin $i$ to destination $j$ averaged over the time duration $[k,k+\tau)$, where $\tau$ denotes the measurement interval, and $1\leq i,j\leq N$ and $1\leq k\leq T$. After analyzing the spatial-temporal features in the traffic data, Zhou et al \cite{ZZXC15} applied the tensor completion method to recover the traffic data from partial measurements and loss. With the help of tensor CP decomposition, an optimization model with spatial-temporal constraints was proposed, whose form can be expressed as
 \begin{equation}\label{ZMMC}
\begin{array}{ll}
\displaystyle{\rm min}_{A,B,C}&\|\mathcal{W}\ast([[A,B,C]]-\mathcal{G}\|_F^2+\lambda(\|A\|_F^2+\|B\|_F^2+\|C\|_F^2)\\
&+\alpha(\|[[FA,B,C]]\|_F^2+\|[[A,GB,C]]\|_F^2+\|[[A,B,HC]]\|_F^2),
\end{array}
\end{equation}
 where $[[\cdot,\cdot,\cdot]]$ is a shorthand notation of CP decomposition of third tensors, $F$ and $G$ are the spatial constraint matrices and $H$ is the temporal constraint matrix. 
 However, such a third order tensor model cannot fully exploit the traffic periodicity in the traffic data, so the recovery accuracy is not very high. To fully exploit the traffic features of periodicity pattern, Xie et al \cite{XWWXWZCZ18} modeled the considered traffic data as another type of third order tensor $\mathcal{Z}\in\mathbb{R}^{o\times t\times d}$, where $o$ corresponds to $N\times N$ OD pairs, and there are $d$ days to consider with each day having $t$ time intervals. It is obvious that $o=N^2$ and $T=td$. In that paper, Xie et al further proposed two sequential tensor completion algorithms to recovery internet traffic data. A public traffic trace, Abilene trace data \cite{Abilene04}, was used as an example to illustrate the model established in \cite{XWWXWZCZ18}, the related traffic data was modeled as a tensor in $\mathbb{R}^{144\times 288\times 168}$. Since its sizes of three dimensions are much more balanced, the authors found that this model is better than the one in \cite{ZZXC15} for missing data recovery. However, in the case where the balance of the above three dimensions is not satisfied (e.g., $N$ is much bigger than $12$, and $t$ and $ d$ remain unchanged), the recovery performance of the corresponding models may be greatly affected.

Existing methods for missing internet network data, whether they are matrix or tensor based methods, sometimes, even when the sampling rate is not too low, the relative error rate of data recovery is still relatively high. Therefore, although various studies have been made to recover missing internet traffic data, how to choose an appropriate tensor to represent traffic data, and how to establish a related optimization model and design an efficient algorithm to solve the established model, are still main challenges in the areas of network management and internet traffic data analysis.
On the other hand, tensor t-product methods introduced by Kilmer et al \cite{KM11,KBHH13} is a powerful tool for decomposing third-order tensors, and has been well applied in  image and video inpainting, for example, see \cite{YHHH16,ZLLZ18}. However, due to the appearance of spatio-temporal features that must be considered in internet data recovery, the models and algorithms used in \cite{YHHH16,ZLLZ18} cannot be directly applied to internet data recovery.

In this paper, we present a new third-order tensor optimization method to recover missing internet traffic data. We first use a third order tensor $\mathcal{F}\in \mathbb{R}^{t\times d\times o}$ to represent the traffic data, where the means of $t$, $d$ and $o$ are stated above, i.e., $t$ represents the numbers of the observing time intervals in each day, $d$ is the number of days considered, and $o=N^2$ where two $N$ correspond respectively to the origin and destination. Based upon the third-order tensor mentioned above, we use t-product of tensors as a useful tool to establish a tensor completion model for internet traffic data recovery.  The temporal stability and periodicity characteristics  of the considered original traffic data are used to improve the established model, and the resulting optimization model has a good separation structure, which is conducive to the design of parallel algorithms and improves efficiency.   While tensor recovery methods based upon t-product decomposition are already very successful in the fields of image, video data recovery, etc.  \cite{YHHH16,ZLLZ18},  what we propose in this paper is the first spatio-temporal tensor t-product method for recovering internet traffic data.

The paper is organized as follows. We present the notation and preliminaries of tensors in Section \ref{Prelim}, which will be used throughout the paper. In Section \ref{TMRef}, we model the traffic data as a third order tensor and formulate the traffic data recovery problem as a low-rank tensor completion model. After an equivalence transformation, a separable optimization model is presented. Furthermore, a modified version of the block coordinate gradient method for missing internet network traffic data is proposed, and the convergence of this algorithm to a stationary point is analyzed in Section \ref{Algorithm}. In Section \ref{Experiments}, we conduct extensive simulation experiments to evaluate the performance of the proposed algorithm. Our experiments are performed on two real-world traffic datasets, the first is the Abilene traffic dataset, and the second is the G\'{E}ANT traffic dataset. The simulation results demonstrate that our model and algorithm can achieve significantly better performance compared with tensor and matrix completion algorithms in the literature, even when the data missing ratio is high. Conclusions and future work are discussed in Section \ref{Conclusion}.

\section{Notation and preliminaries}\label{Prelim}
In this section, we present the notation and some basic preliminaries related to the tensor, tensor t-product and real value functions of complex variables, which will be used in this paper.
\subsection{Notations}

In this paper, the fields of real numbers and complex numbers are denoted as $\mathbb{R}$ and $\mathbb{C}$, respectively. The $n\times n$ identity matrix is denoted by $I_n$. For an $m\times n$ matrix $A$, a subset $I$ of $[m]$ and a subset $J$ of $[n]$, we use the notation $A_{IJ}$ for the $|I|\times |J|$ sub-matrix obtained by deleting all rows $i\not\in I$ and all columns $j\not\in J$, and use the notation $A_{I\cdot}~(A_{\cdot J})$ for the $|I|\times n~(m\times |J|)$ sub-matrix obtained by deleting all rows $i\not\in I$ (all columns $j\not\in J$).  And the spectral norm for a given matrix $A$ is denoted by $\|A\|_2$. 
In what follows, the nomenclatures and the notations in \cite{KB09} on tensors are partially adopted. A tensor is a multidimensional array, and the order of a tensor is the number of dimensions, also called way or mode. For example, a vector is a first-order tensor, and a matrix is a second-order tensor. For the sake of brevity, in general, tensors of order $d\geq 3$ are denoted by Euler script letters $(\mathcal{A}, \mathcal{B},\ldots)$, matrices by capital letters $(A, B,\ldots)$, vectors by bold-case lowercase letters $({\bf a}, {\bf b},\ldots)$, and scalars by lowercase letters $(a,b,\ldots)$. More specifically, a real (complex) tensor of order $d\geq 3$ is represented by $\mathcal{A}\in \mathbb{R}^{m_1\times m_2\times \cdots \times m_d}~(\mathbb{C}^{m_1\times m_2\times \cdots \times m_d})$, and its $(i_1,i_2,\ldots,i_d)$-th element is represented by $a_{i_1i_2\ldots i_d}$. For any two tensors $\mathcal{A},\mathcal{B}\in \mathbb{C}^{m_1\times m_2\times \cdots \times m_d}$, the inner product between $\mathcal{A}$ and $\mathcal{B}$ is denoted as $\langle\mathcal{A},\mathcal{B}\rangle:=\sum_{i_1,i_2,\ldots,i_d}a_{i_1i_2\ldots i_d}\bar b_{i_1i_2\ldots i_d}$ where $\bar b_{i_1i_2\ldots i_d}$ is the conjugate of $b_{i_1i_2\ldots i_d}\in \mathbb{C}$ for $i_j\in[m_j]:=\{1,2,\ldots,m_j\}$ and $j\in[d]$, and the Frobenius norm associated with the above inner product is $\|\mathcal{A}\|=\sqrt{\langle\mathcal{A},\mathcal{A}\rangle}$.

For a given $d$-th order tensor $\mathcal{A}=(a_{i_1i_2\ldots i_d})\in \mathbb{C}^{m_1\times m_2\times \cdots \times m_d}$, it has $d$ modes, namely, mode-$1$, mode-$2$, $\ldots$, mode-$d$. For $k\in[d]$, denote the mode-$k$ matricization (or unfolding) of tensor $\mathcal{A}$ to be ${\rm unfold}(\mathcal{A},k)$, then the $(i_1,i_2,\ldots,i_d)$-th entry of tensor $\mathcal{A}$ is mapped to the $(i_k,j)$-th entry of matrix ${\rm unfold}(\mathcal{A},k)\in \mathbb{C}^{m_k\times \Pi_{l\neq k}m_l}$, where
$$
j=1+\sum_{1\leq l\leq d,l\neq k} (i_l-1)J_l~~~~{\rm with}~~~\displaystyle J_l=\prod_{1\leq t\leq l-1,t\neq k}m_t.
$$
The corresponding inverse operator is denoted as ``fold'', i.e., $\mathcal{A}={\rm fold}({\rm unfold}(\mathcal{A},k),k)$.

In the third order tensor case, we use the terms horizontal, lateral, and frontal slices to specify which two indices are held constant. For a given $\mathcal{A}\in \mathbb{C}^{m_1\times m_2\times m_3}$, using Matlab notation, $\mathcal{A}(i_1,:,:)$ corresponds to the $i_1$-th horizontal slice for $i_1\in [m_1]$, $\mathcal{A}(:,i_2,:)$ corresponds to the $i_2$-th lateral slice for $i_2\in [m_2]$, and $\mathcal{A}(:,:,i_3)$ corresponds the $i_3$-th frontal slice for $i_3\in [m_3]$. More compactly, $A_{i_3}$ is used to represent 
$\mathcal{A}(:,:,i_3)$. 
It is easy to verify that ${\rm unfold}(\mathcal{A},1)=[A_1,A_2,\ldots,A_{m_3}]$ and ${\rm unfold}(\mathcal{A},2)=[A_1^\top,A_2^\top,\ldots,A_{m_3}^\top]$ for $\mathcal{A}\in\mathbb{C}^{m_1\times m_2\times m_3}$.

\subsection{t-product and t-SVD of tensors}
For a given third order tensor $\mathcal{A}\in\mathbb{R}^{m_1\times m_2\times m_3}$  with $m_1\times m_2$ frontal slices, we define the block circulant matrix ${\rm bcirc}(\mathcal{A})\in \mathbb{R}^{m_1m_3\times m_2m_3}$ as
$$
{\rm bcirc}(\mathcal{A}):=\left[
\begin{array}{cccccc}
A_1&A_{m_3}&A_{m_3-1}&\cdots&A_3&A_2\\
A_2&A_1&A_{m_3}&\cdots&A_4&A_3\\
A_3&A_2&A_1&\cdots&A_5&A_4\\
\vdots&\vdots&\vdots&\ddots&\vdots&\vdots\\
A_{m_3-1}&A_{m_3-2}&A_{m_3-3}&\cdots&A_1&A_{m_3}\\
A_{m_3}&A_{m_3-1}&A_{m_3-2}&\cdots&A_2&A_1\\
\end{array}
\right].
$$
Notice that, in this paper, we will always assume the block circulant matrix is created from the frontal slices, and thus there should be no ambiguity with the following notation.

We anchor the ``bvec'' command to the frontal slices of the tensor, that is, ${\rm bvec}(\mathcal{A})$ takes an $m_1\times m_2\times m_3$ tensor and returns a block $m_1m_3\times m_2$ matrix ${\rm bvec}(\mathcal{A}):=[A_1^\top,A_2^\top,\ldots,A_{m_3}^\top]^\top,
$
and the corresponding  inverse operation is defined as ${\rm bvfold}({\rm bvec}(\mathcal{A})):=\mathcal{A}$. Moreover, the block diagonalization operation and its inverse operation are defined as ${\rm bdiag}(\mathcal{A}):={\rm diag}(A_{(1)},A_{(2)},\ldots,A_{(m_3)})\in \mathbb{R}^{m_1m_3\times m_2m_3}$ and ${\rm bdfold}({\rm bdiag}(\mathcal{A})):=\mathcal{A}$, respectively.

\begin{definition}\label{T-ProdDef}(\cite{KBHH13}, Definition 2.5) (t-product) Let $\mathcal{A}\in R^{m_1\times t\times m_3}$ and $\mathcal{B}\in R^{t\times m_2\times m_3}$ be two real tensors. Then the t-product $\mathcal{A}*\mathcal{B}$ is an $m_1\times m_2\times m_3$ real tensor defined by
$$
\mathcal{A}*\mathcal{B}:={\rm bvfold}({\rm bcirc}(\mathcal{A})\cdot{\rm bvec(\mathcal{B})}),
$$
where ``$\cdot$'' means standard matrix product.
\end{definition}


Write $\omega=e^{2\pi {\bf i}/m_3}$ with ${\bf i}=\sqrt{-1}$. It is clear that 
\begin{equation}\label{lsm3}
\displaystyle\sum_{k=1}^{m_3}\omega^{(k-1)(l-1)}\bar \omega^{(k-1)(s-1)}=\left\{
\begin{array}{ll}
m_3,& l=s,\\
0,& otherwise
\end{array}
\right.
\end{equation}
for any integers $l,s\in [m_3]$, where $\bar \omega$ is the conjugate complex number of $\omega$.
Denote
$$
F_{m_3}=\displaystyle\frac{1}{\sqrt{m_3}} \left[
\begin{array}{cccccc}
1&1&1&\cdots&1&1\\
1&\omega&\omega^2&\cdots&\omega^{m_3-2}&\omega^{m_3-1}\\
1&\omega^2&\omega^4&\cdots&\omega^{2(m_3-2)}&\omega^{2(m_3-1)}\\
\vdots&\vdots&\vdots&\ddots&\vdots&\vdots\\
1&\omega^{m_3-2}&\omega^{2(m_3-2)}&\cdots&\omega^{(m_3-2)(m_3-2)}&\omega^{(m_3-2)(m_3-1)}\\
1&\omega^{m_3-1}&\omega^{2(m_3-1)}&\cdots&\omega^{(m_3-2)(m_3-1)}&\omega^{(m_3-1)(m_3-1)}\\
\end{array}
\right],
$$
which is called the normalized discrete Fourier transform (DFT) matrix, and denote the conjugate transpose of $F_{m_3}$ by  $F_{m_3}^*$. 
Just as circulant matrices can be diagonalized by the DFT \cite{GV13}, block-circulant matrices can be block diagonalized, i.e.,
\begin{equation}\label{BDFT}
(F_{m_3}\otimes I_{m_1})\cdot{\rm bicric}(\mathcal{A})\cdot(F^*_{m_3}\otimes I_{m_2})=\tilde A,
\end{equation}
where ``$\otimes$'' denotes the Kronecker product, and $\tilde A={\rm diag}(\tilde A_1,\tilde A_2,\ldots,\tilde A_{m_3})$ with $\tilde A_k$ being
\begin{equation}\label{bar BA-k}
\tilde A_{k}=\sum_{l=1}^{m_3}\omega^{(k-1)(l-1)}A_l,~~  \ \ \forall~ k\in [m_3].
\end{equation}
\begin{remark}\label{blockcon} It should be noticed that, for any $\mathcal{A}\in \mathbb{R}^{m_1\times m_2\times m_3}$, which can be block diagonalized as (\ref{BDFT}), most of the matrices $\tilde{A}_k~(k\in [m_3])$ may be complex, even when $\mathcal{A}$ is symmetric, 
and they satisfy the relationships:
\begin{equation}\label{rcm_3-k+2}
\tilde{A}_1\in \mathbb{R}^{m_1\times m_2}~~~{\rm and}~~~ \tilde{A}_k=\overline{\tilde{A}_{m_3-k+2}},~~ {\rm for~ any~} k\in [m_3]\backslash \{1\}.
\end{equation} 

On the other hand, it is easy to see that any $m_3$ tensors $\tilde{A}_k\in \mathbb{C}^{m_1\times m_2} (k\in [m_3])$, which satisfy the above relationships, can lead to $m_3$ real tensors by the inverse operation of (\ref{bar BA-k}). This fact can be seen from the following formula
\begin{equation}\label{barAAk}
A_{k}=\frac{1}{m_3}\sum_{l=1}^{m_3}\bar\omega^{(k-1)(l-1)}\tilde A_l,~~  \ \ \forall~ k\in [m_3].
\end{equation}
In fact, since $\omega^{nm_3+i}=\omega^i=\bar \omega^{m_3-i}$ for $i$ and $n\in [m_3-1]\cup \{0\}$, we can verify that, from any given $\tilde{A}_k\in \mathbb{C}^{m_1\times m_2}$ for $k\in [m_3]$, the tensors $A_k~(k\in [m_3])$ obtained by (\ref{barAAk}) are real, if and only if (\ref{rcm_3-k+2}) holds.
 \end{remark}

From Definition \ref{T-ProdDef}, it is easy to verify that the t-product of two tensors $\mathcal{C}=\mathcal{A}*\mathcal{B}$ is equivalent to $\tilde C=\tilde A\tilde B$, where $\tilde A={\rm diag}(\tilde A_1,\tilde A_2,\ldots,\tilde A_{m_3})$, $\tilde B={\rm diag}(\tilde B_1,\tilde B_2,\ldots,\tilde B_{m_3})$ and $\tilde C={\rm diag}(\tilde C_1,\tilde C_2,\ldots,\tilde C_{m_3})$, which are defined by (\ref{bar BA-k}), respectively. Moreover, by (\ref{BDFT}), it holds that $\|{\rm bicric}(\mathcal{A})\|_F=\|\tilde A\|_F$.

The identity tensor $\mathcal{I}$ of size $m_1\times m_1\times m_3$ is a tensor whose first frontal slice is a $m_1\times m_1$ identity matrix, and all other frontal slices are zero matrices. The conjugate transpose of a tensor $\mathcal{A}\in \mathbb{R}^{m_1\times m_2\times m_3}$ is a tensor in $\mathbb{R}^{m_2\times m_1\times m_3}$, denoted by $\mathcal{A}^*$, whose each frontal slices are conjugate transposed and then the order of frontal slices are reversed. A f-diagonal tensor is a tensor whose frontal slices are all diagonal matrices. A third order tensor $\mathcal{Q}$ is orthogonal, if $\mathcal{Q}*\mathcal{Q}^*=\mathcal{Q}^**\mathcal{Q}=\mathcal{I}$.

\begin{definition}\label{T-SVD}\cite{KM11} (t-SVD) A tensor $\mathcal{A}\in \mathbb{R}^{m_1\times m_2\times m_3}$ can be factored as $\mathcal{A}=\mathcal{U}\mathcal{S}\mathcal{V}^*$ where $\mathcal{U}\in \mathbb{R}^{m_1\times m_1\times m_3}$ and $\mathcal{V}\in \mathbb{R}^{m_2\times m_2\times m_3}$ are orthogonal tensors , and $\mathcal{S}\in \mathbb{R}^{m_1\times m_2\times m_3}$ is a f-diagonal tensor.
\end{definition}
For any $\mathcal{A} \in\mathbb{R}^{m_1\times m_2\times m_3}$, its tubal rank ${\rm rank}_t(A)$ is defined as the number of nonzero singular tubes of $\mathcal{S}$, i.e., ${\rm rank}_t(A)=\sharp\{i~|~\mathcal{S}(i,i,:)=0\}={\rm max}\{r_1,\ldots,r_{m_3}\}$, where $\mathcal{S}$ is from the t-SVD of $\mathcal{A}=\mathcal{U}\mathcal{S}\mathcal{V}^*$, and $r_k={\rm rank}(\tilde A_k)$ for $k\in [m_3]$.

\begin{proposition}\cite {ZLLZ18} \label{tubrank}
For any tensors $\mathcal{F} \in\mathbb{R}^{m_1\times m_2\times m_3}$, $\mathcal{A} \in\mathbb{R}^{m_1\times m_2\times m_3}$ and $\mathcal{B} \in\mathbb{R}^{m_2\times m_4\times m_3}$, the following properties hold.

(1) If ${\rm rank}_t(\mathcal{F})=r$, then $\mathcal{F}$ can be written into a tensor product form $\mathcal{F}=\mathcal{C}*\mathcal{H}$, where $\mathcal{C} \in\mathbb{R}^{m_1\times r\times m_3}$ and $\mathcal{H} \in\mathbb{R}^{r\times m_2\times m_3}$ are two tensors of smaller sizes and they meet ${\rm rank}_t(\mathcal{C})={\rm rank}_t(\mathcal{H})=r$;

(2) ${\rm rank}_t(\mathcal{A}*\mathcal{B})\leq \min\{({\rm rank}_t(\mathcal{A}),{\rm rank}_t(\mathcal{B})\}$.
\end{proposition}

\subsection{Real value functions of complex variable}
 For ${\bf z}=(z_1,z_2,\ldots,z_n)^\top\in \mathbb{C}^n$, denote
$$
{\bf z}_{re}:=((z_1)_{re},(z_2)_{re},\ldots, (z_n)_{re})^\top
$$
and
$$
{\bf z}_{im}:=((z_1)_{im},(z_2)_{im},\ldots, (z_n)_{im})^\top,
$$
where $(z_i)_{re}$ and $(z_i)_{im}$ are the real and imaginary part of $z_i\in \mathbb{C}$, respectively, for $i\in[n]$. We consider  a real-valued function $f:\mathbb{C}^n\rightarrow \mathbb{R}$ defined by $f({\bf z})=g({\bf z}_{re}, {\bf z}_{im})$, where ${\bf z}={\bf z}_{re}+{\bf i}~ {\bf z}_{im}$.  The following theorem can be found in \cite{B83}.

\begin{theorem}\label{PDconj}
Let $f: \mathbb{C}^n\times \mathbb{C}^n\rightarrow \mathbb{R}$  be a function of a complex vector ${\bf z}$ and its conjugate vector $\bar {\bf z}$, and let $f$ be analytic with respect to each variable (${\bf z}$ and $\bar {\bf z}$) independently. Let $g: \mathbb{R}^n\times \mathbb{R}^n\rightarrow \mathbb{R}$ be the function of the real variables ${\bf z}_{re}$ and ${\bf z}_{im}$ such that $f({\bf z},{\bf z})=g({\bf z}_{re},{\bf z}_{im})$. Then the partial derivative $\nabla_{{\bf z}}f$ (treating $\bar{\bf z}$ as a constant in $f$) gives the same result (on substituting for $\bf z$) as $(\nabla _{{\bf z}_{re}}g-{\bf i}\nabla _{{\bf z}_{re}}g)/2$. Similarly, $\nabla_{\bar{\bf z}}f$  is equivalent to $(\nabla _{{\bf z}_{re}}g+{\bf i}\nabla _{{\bf z}_{re}}g)/2$. Moreover, either of the conditions $\nabla _{\bf z}f={\bf 0}$ or $\nabla _{\bar {\bf z}}f={\bf 0}$ is necessary and sufficient to determine a stationary point of $f$.
\end{theorem}

\section{Tensor model of internet traffic data completion and its reformulation}\label{TMRef}
\subsection{Model}\label{Mod}
Let $\mathcal{G}=(g_{i_1i_2i_3})$ be a given incomplete tensor in $\mathbb{R}^{m_1\times m_2\times m_3}$, and $\Omega$ be the index set of known entries of $\mathcal{G}$, i.e., its entries $g_{i_1i_2i_3}$ are given for $(i_1,i_2,i_3)\in \Omega$ while $g_{i_1i_2i_3}$ are missing for $(i_1,i_2,i_3)\not\in \Omega$. Here, $m_1$, $m_2$ and $m_3$ are corresponding to node $t$, $d$ and $o$ of the collected internet network traffic data in Section \ref{Introd}, respectively. We model the traffic data in an internet network by a tensor $\mathcal{F}\in\mathbb{R}^{m_1\times m_2\times m_3}$. Based on the incomplete tensor $\mathcal{G}$, in order to recover internet network traffic data (i.e., $\mathcal{F}$), we consider the following low rank tensor optimization problem
\begin{equation}\label{Tubrank-optim}
\begin{array}{rl}
\displaystyle\min_{\mathcal{F}}&{\rm rank}_t(\mathcal{F})\\
\text{s.t.}&P_{\Omega}(\mathcal{F})=P_{\Omega}(\mathcal{M}),
\end{array}
\end{equation}
where ${\rm rank}_t (\mathcal{F})$ denotes the tensor tubal rank of $\mathcal{F}$, and $P_{\Omega}$ is the linear operator to extract known elements in the subset $\Omega$ and fills the elements that are not in $\Omega$ with zero values.

The target of this paper is to estimate those elements with zero values as accurately as possible. To solve this problem, a tensor factorization method was proposed to optimize this problem by factorizing tensor $\mathcal{F}\in \mathbb{R}^{m_1\times m_2\times m_3}$ into two smaller tensors $\mathcal{X}\in \mathbb{R}^{m_1\times r\times m_3}$ and $\mathcal{Y}\in \mathbb{R}^{r\times m_2\times m_3}$, where $r$ is the beforehand estimated tubal rank of $\mathcal{F}$ and is usually much smaller than ${\min}\{m_1,m_2\}$.  According to Proposition \ref{tubrank}, we further approximate the model (\ref{Tubrank-optim}) by introducing an intermediate variable $\mathcal{W}\in \mathbb{R}^{m_1\times m_2\times m_3}$ as follows
\begin{equation}\label{ETP-optim}
\begin{array}{rl}
\displaystyle\min_{\mathcal{X},\mathcal{Y},\mathcal{W}}&\displaystyle\frac{1}{2}\|\mathcal{X}*\mathcal{Y}-\mathcal{W}\|_F^2+\frac{\mu}{2}\|\mathcal{W}\|_F^2\\
\text{s.t.}&P_{\Omega}(\mathcal{W})=P_{\Omega}(\mathcal{G})\\
&\mathcal{X}\in \mathbb{R}^{m_1\times r\times m_3},\mathcal{Y}\in \mathbb{R}^{r\times m_2\times m_3},\mathcal{W}\in \mathbb{R}^{m_1\times m_2\times m_3}.
\end{array}
\end{equation}
where $\mu$ is an regularization parameter, which allows a tunable tradeoff between fitting error and achieving tensor low-rank.

\subsection{Improvement with temporal  stability}
In real-world network, most of traffic data often have considerably large difference in start sampling time and end sampling time, but every successive time intervals the sampling data have pretty small difference. That is, traffic usually change slowly over time, which exhibit temporal
stability feature in time dimension. We use matrices $H$ and $K$ to express our knowledge about the traffic temporal properties, in which, $H$ is used to capture the stability of traffic values at two adjacent time slots, and $K$ is used to express the periodicity of traffic data, i.e.,  the similarity in internet visiting behaviors at the same time of different days, such as the similar traffic mode in working hours and sleeping hours. 

We first consider for setting $H$. The temporal constraint matrix $H$ captures temporal stability feature of the traffic tensor, i.e., the traffic data is similar at adjacent time slots. Based on time dimension (mode-$1$) unfolding matrix ${\rm unfold}(\mathcal{X}*\mathcal{Y},1)$ or ${\rm unfold}(\mathcal{W},1)$, a simple choice for the temporal constraint matrix is $H= {\rm Toeplitz}(0,1,-1)$ of the size $(m_1-1)\times m_1$, which denotes the Toeplitz matrix with central diagonal given by ones, and the first upper diagonal given by negative ones, i.e.,
\begin{equation}\label{Toeplitz}
H=\left[
\begin{array}{cccccc}
1&-1&0&\cdots&0&0\\
0&1&-1&\cdots&0&0\\
0&0&1&\cdots&0&0\\
\vdots&\vdots&\vdots&\ddots&\vdots&\vdots\\
0&0&0&\cdots&1&-1
\end{array}
\right]_{(m_1-1)\times m_1}.
\end{equation}

By using the temporal constraint matrix $H$ mentioned above, we can check that
$$
\begin{array}{l}
\left\|H\cdot{\rm unfold}(\mathcal{X}*\mathcal{Y},1)\right\|_F^2\\
\displaystyle=\sum_{1\leq i_1\leq m_1-1,1\leq i_2\leq m_2,1\leq i_3\leq m_3}
\big((\mathcal{X}*\mathcal{Y})_{i_1i_2i_3}-\left(\mathcal{X}*\mathcal{Y}\right)_{(i_1+1)i_2i_3}\big)^2,
\end{array}$$
which means that by minimizing $\|H\cdot{\rm unfold}(\mathcal{X}*\mathcal{Y},1)\|_F^2$ or $\|H\cdot{\rm unfold}(\mathcal{W},1)\|_F^2$, we seek an approximation that also has the property of having similar temporally adjacent values, i.e, the fact that traffic tensor at adjacent points in time of same day are often similar. Similarly, we use a simple Toeplitz matrix
$$
K=\left[
\begin{array}{cccccc}
1&-1&0&\cdots&0&0\\
0&1&-1&\cdots&0&0\\
0&0&1&\cdots&0&0\\
\vdots&\vdots&\vdots&\ddots&\vdots&\vdots\\
0&0&0&\cdots&1&-1
\end{array}
\right]_{(m_2-1)\times m_2}
$$
to capture the periodicity of traffic data. Therefore, by minimizing $\|H\cdot{\rm unfold}(\mathcal{X}*\mathcal{Y},1)\|_F^2+\|K\cdot{\rm unfold}(\mathcal{W},2)\|_F^2$, we may approximate the temporal stability feature of $\mathcal{X}*\mathcal{Y}$ (i.e., $\mathcal{W}$) and are expected to improve recovery accuracy.

Based on the argument above, the model (\ref{ETP-optim}) can be improved into the following optimization model
\begin{equation}\label{TETP-optim}
\begin{array}{rl}
\min&\displaystyle f_0(\mathcal{X},\mathcal{Y},\mathcal{W}):=\frac{1}{2}\|\mathcal{X}*\mathcal{Y}-\mathcal{W}\|_F^2+\frac{\mu}{2}\|\mathcal{W}\|_F^2\\
&+\displaystyle\frac{\rho_1}{2}\|H\cdot{\rm unfold}(\mathcal{X}*\mathcal{Y},1)\|_F^2+\frac{\rho_2}{2}\|K\cdot{\rm unfold}(\mathcal{W},2)\|_F^2\\
\text{s.t.}&P_{\Omega}(\mathcal{W})=P_{\Omega}(\mathcal{G})\\
&\mathcal{X}\in \mathbb{R}^{m_1\times r\times m_3},\mathcal{Y}\in \mathbb{R}^{r\times m_2\times m_3},\mathcal{W}\in \mathbb{R}^{m_1\times m_2\times m_3},
\end{array}
\end{equation}
where $\rho_1$ and $\rho_2$ are two appropriately chosen penalty parameters.
\subsection{Reformulation and simplification of model}\label{Reformu}
In this subsection, we continue to consider how to transform the model (\ref{TETP-optim}) into an equivalent form that is conducive to effective algorithm design.  For simplicity, let us write $\mathcal{Z}=\mathcal{X}*\mathcal{Y}$. 
It is obvious that
\begin{equation}\label{bcrWZ}
(F_{m_3}\otimes I_{m_1})\cdot{\rm bcric}(\mathcal{Z}-\mathcal{W})\cdot(F^*_{m_3}\otimes I_{m_2})=\tilde Z-\tilde W,
\end{equation}
where $\tilde Z={\rm diag}(\tilde Z_1,\tilde Z_2,\ldots,\tilde Z_{m_3})$ and $\tilde W={\rm diag}(\tilde W_1,\tilde W_2,\ldots,\tilde W_{m_3})$, which are defined by (\ref{bar BA-k}), respectively. Consequently, by (\ref{bcrWZ}), we have
$$
\|\mathcal{Z}-\mathcal{W}\|_F^2=\frac{1}{m_3}\|{\rm bcric(\mathcal{Z}-\mathcal{W})}\|_F^2=\frac{1}{m_3}\|\tilde Z-\tilde W\|_F^2=\frac{1}{m_3}\sum_{k=1}^{m_3}\|\tilde Z_k-\tilde {W}_k\|_F^2.
$$
Since $\mathcal{Z}=\mathcal{X}*\mathcal{Y}$, which is equivalent to $\tilde Z=\tilde X\tilde Y$, i.e., $\tilde Z_k=\tilde X_k\tilde  Y_k$ for every $k\in [m_3]$, it holds that
\begin{equation}\label{XYZ}
\|\mathcal{X}*\mathcal{Y}-\mathcal{W}\|_F^2=\frac{1}{m_3}\sum_{k=1}^{m_3}\|\tilde X_k\tilde Y_k-\tilde W_k\|_F^2.
\end{equation}
Similarly, we have
\begin{equation}\label{tildeXY}
\|\mathcal{W}\|_F^2=\frac{1}{m_3}\sum_{k=1}^{m_3}\|\tilde W_k\|_F^2.
\end{equation}

Now we consider $\|H\cdot {\rm unfold}(\mathcal{X}*\mathcal{Y},1)\|_F^2$. Since ${\rm unfold}(\mathcal{Z},1)=[Z_1,Z_2,\ldots,Z_{m_3}]$, it holds that
\begin{equation}\label{HunfZ1}
\|H\cdot {\rm unfold}(\mathcal{Z},1)\|_F^2=\|[HZ_1,HZ_2,\ldots,HZ_{m_3}]\|_F^2=\sum_{k=1}^{m_3}\|HZ_k\|_F^2.
\end{equation}
On the other hand, by (\ref{bar BA-k}), we have
\begin{equation}\label{ZbarZ}
\begin{array}{lll}
\displaystyle\sum_{k=1}^{m_3}\|H\tilde{Z}_k\|^2_F&=&\displaystyle\sum_{k=1}^{m_3}{\rm tr}((H\tilde{Z}_k)(H\tilde{Z}_k)^*)\\
&=&\displaystyle\displaystyle\sum_{k=1}^{m_3}\sum_{s,l=1}^{m_3}\omega^{(k-1)(l-1)}\bar \omega^{(k-1)(s-1)}{\rm tr}((HZ_k)(HZ_k)^\top)\\
&=&\displaystyle\displaystyle\sum_{s,l=1}^{m_3}\sum_{k=1}^{m_3}\omega^{(k-1)(l-1)}\bar \omega^{(k-1)(s-1)}{\rm tr}((HZ_k)(HZ_k)^\top).
\end{array}\end{equation}
By (\ref{lsm3}) and (\ref{ZbarZ}), we further obtain
\begin{equation}\label{ZbarZ-1}
\displaystyle\sum_{k=1}^{m_3}\|H\tilde{Z}_k\|^2_F=m_3\displaystyle\sum_{l=1}^{m_3}{\rm tr}((HZ_l)(HZ_l)^\top)=m_3\displaystyle\sum_{l=1}^{m_3}\|HZ_l\|_F^2.
\end{equation}
Consequently, by combining (\ref{HunfZ1}) and (\ref{ZbarZ-1}), it holds that
\begin{equation}\label{HunfZ2}
\|H\cdot {\rm unfold}(\mathcal{Z},1)\|_F^2=\frac{1}{m_3}\sum_{k=1}^{m_3}\|H\tilde{Z}_k\|_F^2=\frac{1}{m_3}\sum_{k=1}^{m_3}\|H\tilde X_k \tilde Y_k\|_F^2.
\end{equation}

Similarly, we have
\begin{equation}\label{KunfZ}
\|K\cdot {\rm unfold}(\mathcal{W},2)\|_F^2=\frac{1}{m_3}\sum_{k=1}^{m_3}\|\tilde{W}_kK^\top\|^2_F=\frac{1}{m_3}\sum_{k=1}^{m_3}\|\tilde W_k K^\top\|^2_F.
\end{equation}
By (\ref{XYZ}), (\ref{tildeXY}), (\ref{HunfZ2}) and (\ref{KunfZ}), we have
\begin{equation}\label{fun-XYW}
\begin{array}{l}
f_0(\mathcal{X},\mathcal{Y},\mathcal{W})\\
\displaystyle=\frac{1}{2m_3}\sum_{k=1}^{m_3}\left\{\|\tilde X_k \tilde Y_k-\tilde W_{k}\|_F^2+\rho_1\|H \tilde X_k \tilde Y_k\|^2_F+\rho_2\|\tilde W_kK^\top\|^2_F+\mu\|\tilde W_k\|^2_F\right\}.
\end{array}\end{equation}

Due to the special structure of equivalent reformulation (\ref{fun-XYW}) of the objective function in (\ref{TETP-optim}), our algorithm, which will be described in the next section, is closely related to the following two type of optimization sub-problems.
\begin{equation}\label{kXYETP-optim20}
\begin{array}{cl}
\displaystyle\min_{\tilde X_k,\tilde Y_k}&\displaystyle\sum_{k=1}^{m_3}h_k(\tilde X_k, \tilde Y_k)\\
\text{s.t.}&\tilde X_k\in \mathbb{C}^{m_1\times r},\tilde{Y}_k\in \mathbb{C}^{r\times m_2}, k\in[m_3],
\end{array}
\end{equation}
where $h_k(\tilde X_k \tilde Y_k)=\frac{1}{2}(\|\tilde X_k \tilde Y_k-\tilde W_{k}\|_F^2+\rho_1\|H \tilde X_k \tilde Y_k\|^2_F)$, and
\begin{equation}\label{WETP-optim20}
\begin{array}{cl}
\displaystyle\min_{W_k}&\displaystyle\sum_{k=1}^{m_3}g_k(W_k)\\
\text{s.t.}&P_{\Omega_k}(W_k)=P_{\Omega_k}(G_k),\\
&W_k\in \mathbb{R}^{m_1\times m_2},k\in [m_3],
\end{array}
\end{equation}
where $g_{k}(W_k)=\frac{1}{2}(\|Z_k-W_{k}\|_F^2+\rho_2\|W_k K^\top\|^2_F+\mu\|W_k\|^2_F)$ and $\Omega_k=\{(i,j)~|~(i,j,k)\in \Omega)\}$. 

Notice that the models (\ref{kXYETP-optim20}) and (\ref{WETP-optim20}) have both very good separable structure. It is clear that the optimal solution set of (\ref{TETP-optim}) is nonempty, which is denoted by $\Xi_*$, and the corresponding optimal value is denoted by $f_{0*}$, since the objective function value is bounded below with zero. Moreover, we see that $(\mathcal{X}_{*},\mathcal{Y}_{*},\mathcal{W}_{*})$ is a optimal solution of (\ref{TETP-optim}), if and only if that $(\tilde{X}_{*},\tilde{Y}_{*})$ is the optimal solution of (\ref{kXYETP-optim20}) with $\tilde{W}_{k}=\tilde{W}_{k*}$ for $k\in[m_3]$, where $\tilde{X}_{*}=(\tilde{X}_{1*},\tilde{X}_{2*},\ldots,\tilde{X}_{m_3*})$, $\tilde{Y}_{*}=(\tilde{Y}_{1*},\tilde{Y}_{2*},\ldots,\tilde{Y}_{m_3*})$ and $\tilde{W}_{*}=(\tilde{W}_{1*},\tilde{W}_{2*},\ldots,\tilde{W}_{m_3*})$ are obtain by (\ref{bar BA-k}) from $X_{k*}~(k\in [m_3])$, $Y_{k*}~(k\in [m_3])$ and $W_{k*}~(k\in [m_3])$ respectively, and $W_{*}=(W_{1*},W_{2*},\ldots,W_{m_3*})$ is the optimal solution of (\ref{WETP-optim20}) with objective function $\sum_{k=1}^{m_3}g_{k*}(W_k)=\sum_{k=1}^{m_3}\frac{1}{2}(\|Z_{k*}-W_{k}\|_F^2+\rho_2\|W_k K^\top\|^2_F+\mu\|W_k\|^2_F)$ with $Z_{k*}=\mathcal{Z}_*(:,:,k)$ and $\mathcal{Z}_*=\mathcal{X}_**\mathcal{Y}_*$. Here, $W_{*}$ and $\tilde W_{*}$ satisfies (\ref{bar BA-k}) or equivalently (\ref{barAAk}). Therefore, how to solving optimization problem (\ref{TETP-optim}) effectively comes down to solving the following two type of sub-problems
\begin{equation}\label{SubkXYETP-optim2}
\begin{array}{cl}
\displaystyle\min_{\tilde X_k,\tilde Y_k}&\displaystyle h_k(\tilde X_k, \tilde Y_k)\\
\text{s.t.}&\tilde X_k\in \mathbb{C}^{m_1\times r},\tilde{Y}_k\in \mathbb{C}^{r\times m_2},
\end{array}
\end{equation}
for $k\in[m_3]$,
and
\begin{equation}\label{SunkWETP-optim2}
\begin{array}{cl}
\displaystyle\min_{W_k}&\displaystyle g_k(W_k)\\
\text{s.t.}&P_{\Omega_k}(W_k)=P_{\Omega_k}(G_k),\\
&W_k\in \mathbb{R}^{m_1\times m_2}
\end{array}
\end{equation}
for $k\in[m_3]$. Notice that (\ref{SubkXYETP-optim2}) and (\ref{SunkWETP-optim2}) are independent, but for different $k$ they have the same form, respectively. In particular, for the subproblem (\ref{SunkWETP-optim2}), once $\Omega$ is given, $\Omega_k$ is unchanged for $k\in[m_3]$, which implies that the simple equality constraints in (\ref{SunkWETP-optim2}) can be eliminated by substituting them into the objective function. At this time, for each $k\in[m_3]$, if the complement set $\Omega_k^C$ of $\Omega_k$ is non-empty, then the objective function in the resulting unconstrained optimization is a strictly convex quadratic function with respect to the variable $(W _k)_{\Omega_k^C}$, and its coefficient matrix is a positive definite matrix independent of $Z_k$. Hence, for every $k\in [m_3]$ and known $Z_k\in \mathbb{R}^{m_1\times m_2}$, $W _k$ is unique optimal solution of (\ref{SunkWETP-optim2}), if and only if the following KKT condition holds
\begin{equation}\label{subquKKT}
\left\{
\begin{array}{ll}
(W_kK_{\mu\rho_2})_{ij}=(Z_k)_{ij},&{\rm if~} (i,j)\not\in \Omega_k,\\
(W_k)_{ij}=(G_k)_{ij},&{\rm otherwise},
\end{array}
\right.
\end{equation}
where $K_{\mu\rho_2}=(1+\mu)I+\rho_2K^\top K$.
Moreover, denote $\Omega_{ik}=\{j\in [m_2]~|~(i,j,k)\in \Omega~\}$ for every $i\in [m_1]$ and $k\in [m_3]$. Corresponding to $\Omega_{ik}$ and $\Omega^C_{ik}$, the $i$th row vector $(W_k)_{i\cdot}$ of $W_k$ is divided into two blocks, denoted as $(W_k)_{i\cdot}=[(W_k)_{i\Omega_{ik}},(W_k)_{i\Omega^C_{ik}}]$ (the elements in $(W_k)_{i\cdot}$ need to be rearranged if necessary), and accordingly, $K$ is divided into a column-block matrix, i.e., $K=[K_{ik1},K_{ik2}]$. Due to the special structure of (\ref{subquKKT}), we can get the following expression of its solution
\begin{equation}\label{Wkiom}
\left\{
\begin{array}{ll}
(W_k)_{i\Omega^C_{ik}}=\big((Z_k)_{i\Omega^C_{ik}}-\rho_2(G_k)_{i\Omega_{ik}}K^\top_{ik1}K_{ik2}\big)\big((1+\mu)I_{|\Omega^C_{ik}|}+\rho_2K^\top_{ik2}K_{ik2}\big)^{-1},\\
(W_k)_{i\Omega_{ik}}=(G_k)_{i\Omega_{ik}}
\end{array}\right.
\end{equation}
for $i\in [m_1]$ and $k\in [m_3]$.

In the next section, we will propose an algorithm for completing internet traffic data, which is based upon these separable characteristics discussed above.

\section{An algorithm for internet traffic tensor completion problem}\label{Algorithm}
\subsection{Algorithmic description}\label{AlgDes}
In this subsection, the method for internet traffic tensor completion problem in this paper is described in detail.

By a direct computation, we know that, for every $k\in [m_3]$, the partial derivatives $\nabla _{\tilde{X}_k}h_k$ and  $\nabla _{\tilde{Y}_k}h_k$ of the function $h_k$ defined in (\ref{kXYETP-optim20}) (respectively, treating the conjugate complex matrices $\overline{\tilde{X}_k}$ and $\overline{\tilde{Y}_k}$  of $\tilde{X}_k$ and $\tilde{Y}_k$ as constant in $h_k$) are
\begin{equation}\label{hkXYDer}
\left\{\begin{array}{lll}
\nabla _{\tilde X_k}h_k(\tilde X_k,\tilde Y_k)&=&(\tilde X_k \tilde Y_k-\tilde W_k)\tilde Y_k^*+\rho_1H^\top H \tilde X_k \tilde Y_k\tilde Y_k^*,\\
\nabla _{\tilde Y_k}h_k(\tilde X_k,\tilde Y_k)&=&\tilde X_k^*(\tilde X_k \tilde Y_k-\tilde W_k)+\rho_1\tilde X_k^*H^\top H \tilde X_k \tilde Y_k.
\end{array}\right.
\end{equation}
Consequently, by Theorem \ref{PDconj}, the KKT system of (\ref{kXYETP-optim20}) can be expressed as follows
\begin{equation}\label{KKTcondition}
\left\{\begin{array}{l}
H_{\rho_1}\tilde X_k (\tilde Y_k\tilde Y_k^*)=\tilde W_k\tilde Y_k^*,\\
(\tilde X_k ^* H_{\rho_1}\tilde X_k)\tilde Y_k=\tilde X_k ^*\tilde W_k,
\end{array}
\right.
~~~k\in [m_3],
\end{equation}
where $H_{\rho_1}=I+\rho_1H^\top H$, and $\tilde W_k~(k\in [m_3])$ have obtained before solving problem (\ref{kXYETP-optim20}).

Inspired by the application of the generalized inverse matrix in the (approximation) solution of general matrix linear equations, we propose the iterative scheme for solving the model (\ref{kXYETP-optim20}) as follows
\begin{equation}\label{updataXYW1}
\left\{
\begin{array}{lll}
\tilde X_k^{(l+1)}&=&\tilde X_k^{(l)}+(\alpha_{kl}/L^{(l)}_{k1})d_{\tilde X_k}^{(l)},\\
\tilde Y_k^{(l+1)}&=&\tilde Y_k^{(l)}+(\beta_{kl}/L^{(l)}_{k2})d_{\tilde Y_k}^{(l)},\\
\end{array}\right. ~~~~~~~k\in [m_3],
\end{equation}
where $\alpha_{kl}$ and $\beta_{kl}$ are the step-sizes in the linear search of $\tilde X_k$ and $\tilde Y_k$ at $l$th iteration, respectively, $L^{(l)}_{k1}=\|H_{\rho_1}\|_2\|\tilde Y_k^{(l)}(\tilde Y_k^{(l)})^*\|_2$, $L^{(l)}_{k2}=\|(\tilde X_k^{(l+1)})^*H_{\rho_1}\tilde X_k^{(l+1)}\|_2$, and
$$
\left\{
\begin{array}{lll}
d_{\tilde X_k}^{(l)}&=&-H_{\rho_1}^{-1}\nabla_{\tilde X_k}h_k(\tilde X_k^{(l)},\tilde Y_k^{(l)})\big(\tilde Y_k^{(l)}(\tilde Y_k^{(l)})^*\big)^{+},\\
d_{\tilde Y_k}^{(l)}&=&-\big((\tilde X_k^{(l+1)})^*H_{\rho_1}\tilde X_k^{(l+1)}\big)^{+}\nabla_{\tilde Y_k}h_k(\tilde X_k^{(l+1)},\tilde Y_k^{(l)}).\\
\end{array}
\right.$$
Here, $A^+$ denotes the Moore-Penrose generalized inverse of $A\in\mathbb{C}^{m\times n}$ \cite{M20,P55}.

After obtaining  $\tilde X_k^{(l+1)}$ and $\tilde Y_k^{(l+1)}$ for all $k\in [m_3]$, we further obtain $\tilde Z_k^{(l+1)}=\tilde X_k^{(l+1)}\tilde Y_k^{(l+1)}$, and hence $Z_k^{(l+1)}$ by (\ref{barAAk}) for all $k\in [m_3]$. Then, for every $k\in [m_3]$ we obtain $W_k^{(l+1)}$ by solving the constrained quadratic programming (\ref{WETP-optim20}), or equivalently (\ref{SunkWETP-optim2}) with $Z_k=Z^{(l+1)}_k$. From the obtained $W_k^{(l+1)}~(k\in [m_3])$, the complex matrices $\tilde W_k^{(l+1)}~k\in [m_3]$ are further obtained by (\ref{bar BA-k}).

The above steps are repeated  until the required solution is found. This idea is embodied in our algorithm, which is named TCTF2R Algorithm and described as follows

\noindent\rule{\textwidth}{1pt}
\textbf{TCTF2R Algorithm} (Tensor Completion by Tensor Factorization
with Spatio-Temporal Regularization) \\
\noindent\rule{\textwidth}{0.5pt}
\textbf{Input:} The tensor data $\mathcal{G}\in \mathbb{R}^{m_1\times m_2\times m_3}$, the observed set $\Omega$, the temporal constraint matrices $H,K$, the initialized rank vector $r^0\in\mathbb{R}^{m_3}$, the regularization parameters $\rho_1,\rho_2, \mu>0$, and $\varepsilon=1e-6$.\\
\textbf{Initialize:} $\tilde X^{(0)}=(\tilde X_1^{(0)},\tilde X_2^{(0)},\ldots,\tilde X_{m_3}^{(0)})$, $\tilde Y^{(0)}=(\tilde Y_1^{(0)},\tilde Y_2^{(0)},\ldots,\tilde Y_{m_3}^{(0)})$ and $\tilde W^{(0)}=(\tilde W_1^{(0)},\tilde W_2^{(0)},\ldots,\tilde W_{m_3}^{(0)})$, where $\tilde X_k^{(0)}\in \mathbb{C}^{m_1\times r_k^0}$, $\tilde Y_k^{(0)}\in \mathbb{C}^{r_k^0\times m_2}$ and $\tilde W_k^{(0)}\in \mathbb{C}^{m_1\times m_2}$ for $k\in [m_3]$, and the relationship (\ref{rcm_3-k+2}) for $\tilde X^{(0)}$, $\tilde Y^{(0)}$ and $\tilde W^{(0)}$  hold, respectively.\\
\textbf{While not converge do}\\
1. For every $k\in [m_3]$, fix $\tilde Y_k^{(l)}$ and $\tilde W_k^{(l)}$ to update $\tilde X_k^{(l+1)}$ via the first expression in (\ref{updataXYW1}), and obtain $\tilde X^{(l+1)}=(\tilde X_1^{(l+1)},\tilde X_2^{(l+1)},\ldots,\tilde X_{m_3}^{(l+1)})$.\\
2. For every $k\in [m_3]$, fix $\tilde X_k^{(l+1)}$ and $\tilde W_k^{(l)}$ to update $\tilde Y_k^{(l+1)}$ via the second expression in (\ref{updataXYW1}), and obtain $\tilde Y^{(l+1)}=(\tilde Y_1^{(l+1)},\tilde Y_2^{(l+1)},\ldots,\tilde Y_{m_3}^{(l+1)})$.\\
3. For every $k\in [m_3]$, compute $\tilde Z_k^{(l+1)}=\tilde X_k^{(l+1)}\tilde Y_k^{(l+1)}$, and convert $\tilde Z^{(l+1)}=(\tilde Z_1^{(l+1)},\tilde Z_2^{(l+1)},\ldots,\tilde Z_{m_3}^{(l+1)})$ into $Z^{(l+1)}=(Z_1^{(l+1)},Z_2^{(l+1)},\ldots,Z_{m_3}^{(l+1)})$ via (\ref{barAAk}).\\
4. For every $k\in [m_3]$, update $W_k^{(l+1)}$ by solving (\ref{SunkWETP-optim2}) or (\ref{subquKKT}) with $Z_k=Z_k^{(l+1)}$, and obtain $W^{(l+1)}=(W_1^{(l+1)},W_2^{(l+1)},\ldots,W_{m_3}^{(l+1)})$.\\
5. Adjust the rank vector $r^{l+1}\in \mathbb{R}^{m_3}$, and the sizes of $\tilde X_k^{(l+1)}$ and $\tilde Y_k^{(l+1)}$ for $k\in [m_3]$.\\
6. Check the  termination criterion: $\|\tilde X^{(l+1)}-\tilde X^{(l)}\|_{\infty}\le \varepsilon$, $\|\tilde Y^{(l+1)}-\tilde Y^{(l)}\|_{\infty}\le \varepsilon$ and $\|W^{(l+1)}-W^{(l)}\|_F\le \varepsilon$,  \\
7. Convert $W_k^{(l+1)}$ into $\tilde W_k^{(l+1)}$ via (\ref{bar BA-k}) for every $k\in [m_3]$, and $l\leftarrow l+1$.\\
 \textbf{end while}\\
\textbf{Output:} $\tilde X^{(l+1)}$, $\tilde Y^{(l+1)}$, and $W^{(l+1)}=(W_1^{(l+1)},W_2^{(l+1)},\ldots,W_{m_3}^{(l+1)})$ or $\tilde W^{(l+1)}=(\tilde W_1^{(l+1)},\tilde W_2^{(l+1)},\ldots,\tilde W_{m_3}^{(l+1)})$.\\
\noindent\rule{\textwidth}{1pt}

Notice that, in the implementation of the above algorithm, we only need to solve $[m_3/2]+1$ subproblems of form  (\ref{updataXYW1}). After obtaining the first $[m_3/2]+1$ pair $(\tilde X^{(l+1)}_k,\tilde Y^{(l+1)}_k)$ by solving (\ref{updataXYW1}), other pair $(\tilde X_k,\tilde Y_k)$ can be obtained by the following expression
$$(\tilde X^{(l+1)}_k,\tilde Y^{(l+1)}_k)=(\overline{\tilde{X}^{(l+1)}_{m_3-k+2}},\overline{\tilde{Y}^{(l+1)}_{m_3-k+2}}), ~~\forall~k=[m_3/2]+2,\ldots m_3.$$  

\begin{proposition}\label{KKProp} Let $D$ be a $(m-1)\times m$ simple Toeplitz matrix defined by (\ref{Toeplitz}).
It holds that $\lambda_{\min}(D^\top D)=0$ and $\lambda_{\max}(D^\top D)\leq 4$, where $\lambda_{\min}(\cdot)$ and $\lambda_{\max}(\cdot)$ are the smallest and largest eigenvalues of the considered matrix, respectively.
\end{proposition}
\begin{proof}
Since $D$ is a simple Toeplitz matrix, it is easy to see that
$$
D^\top D=\left[
\begin{array}{cccccc}
1&-1&0&\cdots &0&0\\
-1&2&-1&\cdots&0&0\\
0&-1&2&\ddots&0&0\\
\vdots&\vdots&\ddots&\ddots&\ddots&\vdots\\
0&0&0&-1&2&-1\\
0&0&0&\cdots&-1&1\\
\end{array}
\right]_{m\times m}.
$$
Consequently, the desired result follows from the well-known Gerschgorin's circle Theorem \cite{HJ94}.
\end{proof}

From (\ref{subquKKT}) and Proposition \ref{KKProp}, we can see that, for given $Z^{(l+1)}_k~(k\in [m_3])$  and the obtained solution $W^{(l+1)}_k~(k\in [m_3])$ of (\ref{subquKKT}) with $Z_k=Z_k^{(l+1)}$, it holds that
$$
\sum_{k=1}^{m_3}\|(W^{(l+1)}_k-Z_k^{(l+1)})_{\Omega_k^C}\|_F\leq (\mu+4\rho_2)\sum_{k=1}^{m_3}\|W^{(l+1)}_k\|_F,
$$
which shows that, we can choose small enough parameters $\mu$ and $\rho_2$ to force the difference between $W_k^{(l+1)}$ and $Z_k^{(l+1)}$ for $k\in [m_3]$, to meet the required accuracy, but the value of $\rho_2$ will affect the degree of considering the periodicity of the traffic data. In order to ensure a better recovery effect and accuracy, the values of the parameters $\mu$ and $\rho_2$ should be selected appropriately in the implementation of TCTF2R Algorithm.

\subsection{Convergence of the algorithm}\label{GradCom}
The algorithm proposed in this paper can be regarded as a modified version of the block coordinate gradient descent method in \cite{BT13}. In this subsection, we show that every accumulation point of the sequence generated by TCTF2R Algorithm satisfies the stationary point systems of (\ref{kXYETP-optim20}) and (\ref{WETP-optim20}). To this end, we first recall the following lemmas, which will be used to analyze convergence properties of TCTF2R Algorithm.

\begin{lemma} \cite{P55}\label{SVPMPI}
Let $A\in \mathbb{C}^{m\times n}$ and $B=UAV^*$, where $U\in \mathbb{C}^{m\times m}$ and $V\in \mathbb{C}^{n\times n}$ are unitary. Then $B^+=VA^+U^*$. In particular, if $USV^*$ is the singular-value decomposition (SVD) of $A$, where $S={\rm diag}(\sigma_{1},\ldots,\sigma_{t},0,\ldots,0)\in \mathbb{C}^{m\times n}$ with $\sigma_{1}\geq \ldots\geq\sigma_{t}>0$, then $A^+=VS^{-1}U^*$,
where $S^{-1}={\rm diag}(\sigma^{-1}_{1},\ldots,\sigma^{-1}_{t},0,\ldots,0)\in \mathbb{C}^{m\times n}$.
\end{lemma}

\begin{lemma}\cite{Ber08}\label{innerAB}
Let $A\in \mathbb{C}^{m\times m}$ be a Hermitian matrix with eigenvalues $\alpha_1\geq\ldots \geq \alpha_m$. Then, it holds that $
\alpha_m{\rm tr}(B)\leq {\rm tr}(AB) \leq \alpha_1{\rm tr}(B)$ for any given positive semidefinite Hermitian matrix $B\in \mathbb{C}^{m\times m}$.
\end{lemma}

\begin{lemma}\cite{TH77}\label{ASHermitian}
Let $C=B^*AB$ with eigenvalues $\lambda_1\geq \ldots\geq\lambda_n$, where $A\in \mathbb{C}^{m\times m}$ is a Hermitian matrix with eigenvalues $\alpha_1\geq\ldots \geq \alpha_m\geq 0$, and $B\in\mathbb{C}^{m\times n}$ with singular values $\beta_1\geq \ldots\geq\beta_n\geq 0$, i.e., the eigenvalues of the positive semidefinite matrix $(B^*B)^{1/2}$. Then the following statements hold

(1) $\lambda_{i+j-1}\leq \beta_i^2\alpha_j$, for every $1\leq i\leq n$, $1\leq j\leq m$ and $i+j-1\leq n$.

(2) $\beta_i^2\alpha_j\leq\lambda_{i+j-m}$, for every $1\leq i\leq n$, $1\leq j\leq m$ and $i+j>m$.
\end{lemma}

In what follows, we focus on discussing the properties of the search directions $d_{\tilde X}^{(l)}$ and $d_{\tilde Y}^{(l)}$. For every $k\in [m_3]$, let the sequence $\{(\tilde X_k^{(l)},\tilde Y_k^{(l)})\}$ be generated by TCTF2R Algorithm. Let $U_{\tilde Y_k}^{(l)}S_{\tilde Y_k}^{(l)}(V_{\tilde Y_k}^{(l)})^*$ be the singular-value decomposition (SVD) of $\tilde Y_k^{(l)}$ with $S_{\tilde Y_k}^{(l)}={\rm diag}(\sigma_{\tilde Y_kl1},\ldots,\sigma_{\tilde Y_klt},0,\ldots,0)\in \mathbb{C}^{r\times m_2}$ and $\sigma_{\tilde Y_kl1}\geq \ldots\geq\sigma_{\tilde Y_klt}>0$, and let $U_{\tilde X_k}^{(l)}S_{\tilde X_k}^{(l)}(V_{\tilde X_k}^{(l)})^*$ be the SVD of $\tilde X_k^{(l+1)}$, where $S_{\tilde X_k}^{(l)}={\rm diag}(\sigma_{\tilde X_kl1},\ldots,\sigma_{\tilde X_kls},0,\ldots,0)\in \mathbb{C}^{m_1\times r}$ with $\sigma_{\tilde X_k}l1\geq \ldots\geq\sigma_{\tilde X_kls}>0$.

\begin{proposition}\label{DirecX}
Let $U_{\tilde Y_k}^{(l)}=[U_{\tilde Y_k1}^{(l)},U_{\tilde Y_k2}^{(l)}]$, where $U_{\tilde Y_k1}^{(l)}$ be corresponding to the positive singular-values of $\tilde Y_k^{(l)}$. Then it holds that
\begin{equation}\label{XDirection}
\langle d_{\tilde X_k}^{(l)}, \nabla_{\tilde X_k}h_k(\tilde X_k^{(l)},\tilde Y_k^{(l)})\rangle\leq \displaystyle-\frac{1}{(1+4\rho_1)\sigma^2_{\tilde Y_kl1}}\|\nabla_{\tilde X_k}h_k(\tilde X_k^{(l)},\tilde Y_k^{(l)})U_{\tilde Y_k1}^{(l)}\|_F^2.
\end{equation}
\end{proposition}

\begin{proof}
For writing concisely, the subscript letters $\tilde Y_k$ and $k$ in this lemma have been removed in our proof. Since $U^{(l)}S^{(l)}(V^{(l)})^*$ is the SVD of $\tilde Y^{(l)}$, we have
$\tilde Y^{(l)}(\tilde Y^{(l)})^*=U^{(l)}\tilde S^{(l)}(U^{(l)})^*$, where $\tilde S^{(l)}=S^{(l)}(S^{(l)})^\top={\rm diag}(\tilde S_1^{(l)},O)$, $\tilde S_1^{(l)}={\rm diag}(\sigma^2_{l1},\ldots,\sigma^2_{lt})\in \mathbb{C}^{t\times t}$ and $O\in \mathbb{C}^{(r-t)\times (r-t)}$ is a zero matrix. Consequently, by Lemma \ref{SVPMPI}, it holds that
\begin{equation}\label{tildeY}
(\tilde Y^{(l)}(\tilde Y^{(l)})^*)^+=U^{(l)}(\tilde S^{(l)})^{+}(U^{(l)})^*,
\end{equation}
where $(\tilde S^{(l)})^{+}={\rm diag}((\tilde S_1^{(l)})^{-1},O)$ and $(\tilde S_1^{(l)})^{-1}={\rm diag}(\sigma^{-2}_{l1},\ldots,\sigma^{-2}_{lt})$. For simplicity, we write $R_{U1}^{(l)}=\nabla_{\tilde X}h(\tilde X^{(l)},\tilde Y^{(l)})U^{(l)}_1$ and $R_{U2}^{(l)}=\nabla_{\tilde X}h(\tilde X^{(l)},\tilde Y^{(l)})U^{(l)}_2$. Then, by (\ref{tildeY}),
we know
$$
\begin{array}{lll}
d_{\tilde X}^{(l)}&=&-H_{\rho_1}^{-1}\nabla_{\tilde X}h(\tilde X^{(l)},\tilde Y^{(l)})\big(\tilde Y^{(l)}(\tilde Y^{(l)})^*\big)^{+}\\
&=&-H_{\rho_1}^{-1}[R_{U1}^{(l)},R_{U2}^{(l)}](\tilde S^{(l)})^{+}(U^{(l)})^*,
\end{array}
$$
which implies
\begin{equation}\label{IndXf}
\begin{array}{lll}
\langle d_{\tilde X}^{(l)}, \nabla_{\tilde X}h(\tilde X^{(l)},\tilde Y^{(l)})\rangle&=&-{\rm tr}\left(H_{\rho_1}^{-1}[R_{U1}^{(l)},R_{U2}^{(l)}](\tilde S^{(l)})^{+}[R_{U1}^{(l)},R_{U2}^{(l)}]^*\right)\\
&=&-{\rm tr}\left(H_{\rho_1}^{-1}R_{U1}^{(l)}(\tilde S_1^{(l)})^{-1}(R_{U1}^{(l)})^*\right).
\end{array}
\end{equation}
It is clear that $H_{\rho_1}^{-1}$ and $(\tilde S_1^{(l)})^{-1}$ are positive definite. Consequently, by (\ref{IndXf}), Proposition \ref{KKProp} and Lemma \ref{innerAB}, it holds that
$$
\langle d_{X}^{(l)}, \nabla_{\tilde X}h(\tilde X^{(l)},\tilde Y^{(l)})\rangle\leq -\frac{1}{(1+4\rho_1)\sigma^2_{l1}}{\rm tr}((R_{U1}^{(l)})^*R_{U1}^{(l)})=-\frac{1}{(1+4\rho_1)\sigma^2_{l1}}\|R_{U1}^{(l)}\|_F^2.
$$
From this, we obtain the desired result (\ref{XDirection}), since $R_{U1}^{(l)}=\nabla_{\tilde X}h(\tilde X^{(l)},\tilde Y^{(l)})U^{(l)}_1$. We complete the proof.
\end{proof}

\begin{proposition}\label{DirecY}
 Let $V_{\tilde X_k}^{(l)}=[V_{\tilde X_k1}^{(l)},V_{\tilde X_k2}^{(l)}]$, where $V_{\tilde X_k1}^{(l)}$ is corresponding to the positive singular-values of $\tilde X_k^{(l+1)}$. Then it holds that
\begin{equation}\label{XDirection}
\langle d_{\tilde Y_k}^{(l)},\nabla_{\tilde Y_k}h_k(\tilde X_k^{(l+1)},\tilde Y_k^{(l)})\rangle\leq \displaystyle-\frac{1}{(1+4\rho_1)\sigma^2_{\tilde X_kl1}}\|(V_{\tilde X_k1}^{(l)})^*\nabla_{\tilde Y_k}h_k(\tilde X_k^{(l+1)},\tilde Y_k^{(l)})\|_F^2.
\end{equation}
\end{proposition}
\begin{proof}
For writing concisely, the subscript letters $\tilde X_k$ and $k$ in this lemma have been removed in our proof. Since $U^{(l)}S^{(l)}(V^{(l)})^*$ is the SVD of $\tilde X^{(l+1)}$, we have
$$
(\tilde X^{(l+1)})^*H_{\rho_1}\tilde X^{(l+1)}=V^{(l)}\Sigma^{(l)} (V^{(l)})^*
$$
with $\Sigma^{(l)}=(S^{(l)})^\top (U^{(l)})^*H_{\rho_1}U^{(l)}S^{(l)}$, which implies, together Lemma \ref{SVPMPI}, that
\begin{equation}\label{XIHH}
\big((\tilde X^{(l+1)})^*H_{\rho_1}\tilde X^{(l+1)}\big)^+=V^{(l)}(\Sigma^{(l)})^+ (V^{(l)})^*.
\end{equation}
Write $S^{(l)}={\rm diag}(S_2^{(l)},O)\in \mathbb{C}^{r\times r}$ with $S_2^{(l)}={\rm diag}(\sigma_{l1},\ldots,\sigma_{ls})\in \mathbb{C}^{s\times s}$ and $O\in \mathbb{C}^{(r-s)\times (r-s)}$ being a zero matrix, $U^{(l)}=[U^{(l)}_1,U^{(l)}_2]$ and $V^{(l)}=[V^{(l)}_1,V^{(l)}_2]$, where $U_1^{(l)}$ and $V_1^{(l)}$ are corresponding to $S_2^{(l)}$. Consequently, we have $\Sigma^{(l)}={\rm diag}(\Sigma^{(l)}_1,O)$, which implies
\begin{equation}\label{Sigmal}
(\Sigma^{(l)})^+={\rm diag}((\Sigma^{(l)}_1)^+,O)={\rm diag}((\Sigma^{(l)}_1)^{-1},O),
\end{equation}
 where $\Sigma^{(l)}_1=(S_2^{(l)})^\top (U_1^{(l)})^*H_{\rho_1}U_1^{(l)}S_2^{(l)}$ is invertible, which implies, together with Proposition \ref{KKProp} and Lemma \ref{ASHermitian}, that
 \begin{equation}\label{lammax}
 \lambda_{\min}((\Sigma^{(l)}_1)^{-1})=(\lambda_{\max}(\Sigma^{(l)}_1))^{-1}\geq \frac{1}{(1+4\rho_1)\sigma_{l1}^2}.
 \end{equation}
From the definition of $d_{\tilde Y}^{(l)}$, (\ref{XIHH}), (\ref{Sigmal}) and (\ref{lammax}), we have
 $$
 \begin{array}{lll}
 \langle d_{\tilde Y}^{(l)},\nabla_{\tilde Y}h(\tilde X^{(l+1)},\tilde Y^{(l)})\rangle&=&-{\rm tr}\big((R_{V1}^{(l)})^*(\Sigma^{(l)}_1)^{-1}R_{V1}^{(l)}\big)\\
 &\leq &-\lambda_{\min}((\Sigma^{(l)}_1)^{-1})\|R_{V1}^{(l)}\|_F^2\\
 &\leq&-\displaystyle\frac{1}{(1+4\rho_1)\sigma_{l1}^2}\|R_{V1}^{(l)}\|_F^2,
\end{array}
$$
where $R_{V1}^{(l)}=(V_1^{(l)})^*\nabla_{\tilde Y}h(\tilde X^{(l+1)},\tilde Y^{(l)})$, and the first inequality comes from Lemma \ref{innerAB}. We obtain the desired result and complete the proof.
\end{proof}

\begin{remark}
From Propositions \ref{DirecX} and \ref{DirecY}, we see that, $d_{\tilde X_k}^{(l)}$ is a descent direction of $h_k$ at $(\tilde X_k^{(l)},\tilde Y_k^{(l)})$ provided $\nabla_{\tilde X_k}h_k(\tilde X_k^{(l)},\tilde Y_k^{(l)})U_{\tilde Y_k1}^{(l)}\neq0$, and $d_{\tilde Y}^{(l)}$ is a descent direction of $h_k$ at $(\tilde X_k^{(l+1)},\tilde Y_k^{(l)})$ provided  $(V_{\tilde X_k1}^{(l)})^*\nabla_{\tilde Y}h_k(\tilde X_k^{(l+1)},\tilde Y_k^{(l)})\neq0$. In particular, by Proposition \ref{DirecX}, we know that, if $\tilde Y^{(l)}_k$ is of full row rank, then $U^{(l)}_{\tilde Y_k1}=U_{\tilde Y_k}^{(l)}$ is unitary. In this case, $d_{\tilde X_k}^{(l)}$ is a descent direction of $h_k$ at $(\tilde X_k^{(l)},\tilde Y_k^{(l)})$, provided $\nabla_{\tilde X_k}h_k(\tilde X_k^{(l)},\tilde Y_k^{(l)})\neq0$, which is similar to the conclusion that the negative gradient direction is the descent direction under normal circumstances. By Proposition \ref{DirecY}, we know that, a similar result also holds for $d_{\tilde Y_k}^{(l)}$, i.e., if $\tilde X_k^{(l+1)}$ is of full column rank, then  $d_{\tilde Y_k}^{(l)}$ is a descent direction of $h_k$ at $(\tilde X_k^{(l+1)},\tilde Y_k^{(l)})$, provided $\nabla_{\tilde Y_k}h_k(\tilde X_k^{(l+1)},\tilde Y_k^{(l)})\neq0$. \end{remark}

\begin{remark}
From Propositions \ref{DirecX} and \ref{DirecY}, we know that
\begin{equation}\label{XXYY1}
\begin{array}{l}
\langle \tilde X_k^{(l+1)}-\tilde X_k^{(l)},\nabla_{\tilde X_k}h_k(\tilde X_k^{(l)},\tilde Y_k^{(l)})\rangle\\
~~~~~~\leq \displaystyle-\frac{\alpha_{kl}}{L^{(l)}_{k1}(1+4\rho_1)\sigma^2_{\tilde Y_kl1}}\|\nabla_{\tilde X_k}h_k(\tilde X_k^{(l)},\tilde Y_k^{(l)})U_{Y_k1}^{(l)}\|_F^2.
\end{array}\end{equation}
and \begin{equation}\label{XXYY2}
\begin{array}{l}
\langle \tilde Y_k^{(l+1)}-\tilde Y_k^{(l)},\nabla_{\tilde Y_k}h_k(\tilde X_k^{(l+1)},\tilde Y_k^{(l)})\rangle\\
~~~~~~\leq \displaystyle-\frac{\beta_{kl}}{L^{(l)}_{k2}(1+4\rho_1)\sigma^2_{\tilde X_kl1}}\|(V_{\tilde X_k1}^{(l)})^*\nabla_{\tilde Y_k}h_k(\tilde X_k^{(l+1)},\tilde Y_k^{(l)})\|_F^2.
\end{array}\end{equation}
\end{remark}

\begin{proposition}\label{XX2YY}
Let $U_{\tilde Y_k}^{(l)}=[U_{\tilde Y_k1}^{(l)},U_{\tilde Y_k2}^{(l)}]$ and $V_{\tilde X_k}^{(l)}=[V_{\tilde X_k1}^{(l)},V_{\tilde X_k2}^{(l)}]$, where $U_{\tilde Y_k1}^{(l)}$ and $V_{\tilde X_k1}^{(l)}$ are corresponding to the positive singular-values of $\tilde Y_k^{(l)}$ and $\tilde X_k^{(l+1)}$, respectively. Then it holds that
\begin{equation}\label{XXYYZZ}
\left\{\begin{array}{l}
\|\tilde X_k^{(l+1)}-\tilde X_k^{(l)}\|_F^2\leq \displaystyle\frac{\alpha^2_{kl}}{(L^{(l)}_{k1})^2\sigma^4_{\tilde Y_klt}}\|\nabla_{\tilde X_k}h_k(\tilde X_k^{(l)},\tilde Y_k^{(l)})U_{Y_k1}^{(l)}\|_F^2,\\
\|\tilde Y_k^{(l+1)}-\tilde Y_k^{(l)}\|_F^2\leq \displaystyle\frac{\beta^2_{l}}{(L^{(l)}_{k2})^2\sigma^4_{\tilde X_kls}}\|(V_{\tilde X_k1}^{(l)})^*\nabla_{\tilde Y_k}h_k(\tilde X_k^{(l+1)},\tilde Y_k^{(l)})\|_F^2.
\end{array}
\right.
\end{equation}
\end{proposition}

\begin{proof}
Since $$\tilde X_k^{(l+1)}-\tilde X_k^{(l)}=-(\alpha_{kl}/L_{k1}^{(l)})H_{\rho_1}^{-1}\nabla_{\tilde X_k}h_k(\tilde X_k^{(l)},\tilde Y_k^{(l)})(\tilde Y_k^{(l)}(\tilde Y_k^{(l)})^*)^+,$$ it holds that
$$
\begin{array}{l}
\|\tilde X_k^{(l+1)}-\tilde X_k^{(l)}\|_F^2\\
=\frac{\alpha_{kl}^2}{(L_{k1}^{(l)})^2}\langle H_{\rho_1}^{-1}\nabla_{\tilde X_k}h_k(\tilde X_k^{(l)},\tilde Y_k^{(l)})(\tilde Y_k^{(l)}(\tilde Y_k^{(l)})^*)^+,H_{\rho_1}^{-1}\nabla_{\tilde X_k}h_k(\tilde X_k^{(l)},\tilde Y_k^{(l)})(\tilde Y_k^{(l)}(\tilde Y_k^{(l)})^*)^+\rangle\\
=\frac{\alpha_{kl}^2}{(L_{k1}^{(l)})^2}{\rm tr}((\tilde Y_k^{(l)}(\tilde Y_k^{(l)})^*)^+(\nabla_{\tilde X_k}h_k(\tilde X_k^{(l)},\tilde Y_k^{(l)}))^* H_{\rho_1}^{-2}\nabla_{\tilde X_k}h_k(\tilde X_k^{(l)},\tilde Y_k^{(l)})(\tilde Y_k^{(l)}(\tilde Y_k^{(l)})^*)^+).
\end{array}
$$
Consequently, by Lemma \ref{ASHermitian}, the first inequality in (\ref{XXYYZZ}) follows. The second  inequality can be proved similarly.
\end{proof}

\begin{proposition}\label{inWWk}
For every $k\in [m_3]$ and $Z_k^{(l+1)}$, let $W_k^{(l+1)}$ be the optimal solution of (\ref{SunkWETP-optim2}) with objective function being $g_{k}(W_k)=\frac{1}{2}(\|Z_k^{(l+1)}-W_{k}\|_F^2+\rho_2\|W_k K^\top\|^2_F+\mu\|W_k\|^2_F)$. Then it holds that
$$
f_0(\mathcal{X}^{(l+1)},\mathcal{Y}^{(l+1)},\mathcal{W}^{(l+1)})\leq f_0(\mathcal{X}^{(l+1)},\mathcal{Y}^{(l+1)},\mathcal{W}^{(l)}),
$$
where $\mathcal{W}^{(l+1)}={\rm bvfold}\big(((W_1^{(l+1)})^\top,(W_2^{(l+1)})^\top,\ldots,(W_{m_3}^{(l+1)})^\top)^\top\big)$.
\end{proposition}
\begin{proof}
Since $g_{k}(W^{(l+1)}_k)\leq g_{k}(W^{(l)}_k)$, we have
$$
\begin{array}{lll}
f_0(\mathcal{X}^{(l+1)},\mathcal{Y}^{(l+1)},\mathcal{W}^{(l+1)})&=&\displaystyle\sum_{k=1}^{m_3}g_{k}(W^{(l+1)}_k)+\frac{\rho_1}{2}\sum_{k=1}^{m_3}\|HZ_k^{(l+1)}\|_F^2\\
&\leq&\displaystyle\sum_{k=1}^{m_3}g_{k}(W^{(l)}_k)+\frac{\rho_1}{2}\sum_{k=1}^{m_3}\|HZ_k^{(l+1)}\|_F^2\\
&=&\displaystyle f_0(\mathcal{X}^{(l+1)},\mathcal{Y}^{(l+1)},\mathcal{W}^{(l)}).
\end{array}$$
We obtain the desired result and complete the proof.
\end{proof}
\begin{theorem}\label{GlobConvergM}
 For every $k\in [m_3]$, suppose that $\nabla_{\tilde X_k}h_k(\tilde X_k^{(l)},\tilde Y_k^{(l)})U_{\tilde Y_k1}^{(l)}\neq0$ and $(V_{\tilde X_k1}^{(l)})^*\nabla_{\tilde Y_k}h_k(\tilde X_k^{(l+1)},\tilde Y_k^{(l)})\neq0$ for $l=0,1,2,\ldots$, where $U^{(l)}_{\tilde Y_k1}$ and $V_{\tilde X_k1}^{(l)}$ are defined in Proposition \ref{XX2YY}. Then, the sequences  $\{\tilde{X}_k^{(l)}\tilde{Y}_k^{(l)}\}$ and $\{\tilde{W}_k^{(l)}\}$ are both bounded. Moreover, for every accumulation point $(\hat{W}_{k\star},\tilde{W}_{k\star})$ of $\{(\tilde X_k^{(l)}\tilde Y_k^{(l)},\tilde W_k^{(l)})\}_{l=1}^\infty$, there exists $\tilde X_{k\star}\in \mathbb{R}^{m_1\times r}$ and $\tilde Y_{k\star}\in \mathbb{R}^{r\times m_2}$ such that $\hat{W}_{k\star}=\tilde X_{k\star}\tilde Y_{k\star}$ and $\{f_0(\mathcal{X}^{(l)},\mathcal{Y}^{(l)},\mathcal{W}^{(l)})\}$  monotonically converges to $\tilde f_{\star}:=f(\mathcal{X}_{\star},\mathcal{Y}_{\star},\mathcal{W}_{\star})$ on the whole sequence.
\end{theorem}

\begin{proof} It is obvious that $\{f_0^{(l)}:=f_0(\mathcal{X}^{(l)},\mathcal{Y}^{(l)},\mathcal{W}^{(l)})\}$ is bounded below with zero. Moreover, from the given condition, we know that $d_{\tilde X_k}^{(l)}$ and $d_{\tilde Y}^{(l)}$ are descent directions of $h_k$ at $(\tilde X_k^{(l)},\tilde Y_k^{(l)})$ and  $(\tilde X_k^{(l+1)},\tilde Y_k^{(l)})$, respectively. Hence, it holds that
$$
h_k(\tilde X_k^{(l+1)},\tilde Y_k^{(l+1)})\leq h_k(\tilde X_k^{(l+1)},\tilde Y_k^{(l)})\leq h_k(\tilde X_k^{(l)},\tilde Y_k^{(l)}),
$$
which implies
$$
f_0(\mathcal{X}^{(l+1)},\mathcal{Y}^{(l+1)},\mathcal{W}^{(l)})\leq f_0(\mathcal{X}^{(l)},\mathcal{Y}^{(l)},\mathcal{W}^{(l)}).
$$
By this and Proposition \ref{inWWk}, we know that
$$
f_0(\mathcal{X}^{(l+1)},\mathcal{Y}^{(l+1)},\mathcal{W}^{(l+1)})\leq f_0(\mathcal{X}^{(l)},\mathcal{Y}^{(l)},\mathcal{W}^{(l)}),
$$
that is, the sequence $\{f_0^{(l)}\}$ decreases monotonically. On the other hand, form the structure of $f_0$, it is obvious that the function $f_0$ is coercive with respect to  $\mathcal{W}$, which implies, together with the fact that $\{f_0^{(l)}\}$ is non-increasing, that $\{\tilde W_k^{(l)}\}$ is bounded for every $k\in [m_3]$, and hence $\{\tilde X_k^{(l)}\tilde Y_k^{(l)}\}$ is also bounded. Consequently, there exist subsequence $\{\tilde X_k^{(l_j)}\tilde Y_k^{(l_j)}\}$ of $\{\tilde X_k^{(l)}\tilde Y_k^{(l)}\}$ and subsequence $\{\tilde W_k^{(l_j)}\}$ of $\{\tilde W_k^{(l)}\}$, such that
$$
\tilde X_k^{(l_j)}\tilde Y_k^{(l_j)}\rightarrow \hat{W}_{k\star}~~{\rm and}~~\tilde W_k^{(l_j)}\rightarrow \tilde{W}_{k\star}, ~~~{\rm for }~~k=1,2,\ldots,m_3,
$$
as $l_j\rightarrow \infty$. It is clear that ${\rm rank}(\tilde X_k^{(l_j)})\leq r$ and ${\rm rank}(\tilde Y_k^{(l_j)})\leq r$, which implies ${\rm rank}(\hat{W}_{k\star})\leq r$. Hence, there exist $\tilde{X}_{k\star}\in \mathbb{R}^{m_1\times r}$ and $\tilde{Y}_{k\star}\in \mathbb{R}^{r\times m_2}$ such that $\hat{W}_{k\star}=\tilde{X}_{k\star}\tilde{Y}_{k\star}$. Therefore, exist there $(\mathcal{X}_{\star},\mathcal{Y}_{\star},\mathcal{W}_{\star})$ such that $\{f_0^{(l)}\}\rightarrow  f_{0\star}:=f_0(\mathcal{X}_{\star},\mathcal{Y}_{\star},\mathcal{W}_{\star})$ on the whole sequence. We obtain the desired result and complete the proof.
\end{proof}

In what follows, we study the convergence property of TCTF2R Algorithm. We always assume that for every $k\in [m_3]$, it holds that $\nabla_{\tilde X_k}h_k(\tilde X_k^{(l)},\tilde Y_k^{(l)})U^{(l)}_{\tilde Y_k1}\neq0$ and $(V_{\tilde X_k1}^{(l)})^*\nabla_{\tilde Y_k}h_k(\tilde X_k^{(l+1)},\tilde Y_k^{(l)})\neq0$ for any $l=0,1,\ldots$, where $U^{(l)}_{\tilde Y_k1}$ and $V_{\tilde X_k1}^{(l)}$ are defined in Propositions \ref{DirecX} and \ref{DirecY}, respectively. From (\ref{hkXYDer}), it is easy to see that, for every $k\in [m_3]$ and $(\tilde X_k,\tilde Y_k)\in \mathbb{C}^{m_1\times r}\times \mathbb{C}^{r\times m_2}$, the following inequalities hold.
\begin{equation}\label{XYWLipch}
\left\{\begin{array}{lll}
\|\nabla_{\tilde X_k} h_k(\tilde X_k,\tilde Y_k)-\nabla_{\tilde X_k} h_k(\tilde X_k+\triangle\tilde X,\tilde Y_k)\|_F\leq L_{1}(\tilde Y_k)\|\triangle\tilde X\|_F,~\forall~\triangle\tilde X\in \mathbb{C}^{m_1\times r},\\
\|\nabla_{\tilde Y_k} h_k(\tilde X_k,\tilde Y_k)-\nabla_{\tilde Y_k} h_k(\tilde X_k,\tilde Y_k+\triangle\tilde Y)\|_F\leq L_{2}(\tilde X_k)\|\triangle\tilde Y\|_F,~\forall~\triangle\tilde Y\in \mathbb{C}^{r\times m_1},\\
\end{array}\right.
\end{equation}
where $L_{1}(\tilde Y_k)=\|H_{\rho_1}\|_2\|\tilde Y_k\tilde Y_k^*\|_2$ and $L_{2}(\tilde X_k)=\|\tilde X_k^*H_{\rho_1}\tilde X_k\|_2$.

For the sake of completeness, we first recall the following lemma, which is a block matrix version of the well-known descent lemma \cite{Bert99} and can be proved similarly, since (\ref{XYWLipch}) hold.

\begin{lemma}\label{lipsh}
For every $k\in [m_3]$ and $\tilde W_k\in C^{m_1\times m_2}$, let $h_k: \mathbb{C}^{m_1\times r}\times \mathbb{C}^{r\times m_2}\rightarrow \mathbb{R}$ be defined in (\ref{kXYETP-optim20}). For any given $(\tilde X_k,\tilde Y_k)\in \mathbb{C}^{m_1\times r}\times \mathbb{C}^{r\times m_2}$, we have
$$
\left\{\begin{array}{lll}
h_k(\tilde X_k+\triangle\tilde X_k,\tilde Y_k)\leq h_k(\tilde X_k,\tilde Y_k)+\langle \nabla_{\tilde X_k} h_k(\tilde X_k,\tilde Y_k),\triangle\tilde X_k\rangle+\frac{L_{1}(\tilde Y_k)}{2}\|\triangle\tilde X_k\|_F^2,\\
h_k(\tilde X_k,\tilde Y_k+\triangle\tilde Y_k)\leq h_k(\tilde X_k,\tilde Y_k)+\langle \nabla_{\tilde Y_k} h_k(\tilde X_k,\tilde Y_k),\triangle\tilde Y_k\rangle+\frac{L_{2}(\tilde X_k)}{2}\|\triangle\tilde Y_k\|_F^2
\end{array}\right.$$
for any $(\triangle\tilde X_k, \triangle\tilde Y_k)\in \mathbb{C}^{m_1\times r}\times\mathbb{C}^{r\times m_2}$, where $L_{1}(\tilde Y_k)$, and $L_{2}(\tilde X_k)$ are defined in (\ref{XYWLipch}).
\end{lemma}

\begin{proposition}\label{Prop4}
Let $\{(\tilde X_k^{(l)},\tilde Y_k^{(l)},\tilde W_k^{(l)})\}$ be a sequence generated by TCTF2R Algorithm. Then, for every $k\in [m_3]$, it holds that
\begin{equation}\label{ineqZZl}
\begin{array}{l}
h_k(\tilde X_k^{(l)},\tilde Y_k^{(l)})-h_k(\tilde X_k^{(l+1)},\tilde Y_k^{(l+1)})\\
\geq \kappa_{kl1}\left(\|\nabla_{\tilde X_k}h_k(\tilde X_k^{(l)},\tilde Y_k^{(l)})U_{Y_k1}^{(l)}\|_F^2+\|(V_{\tilde X_k1}^{(l)})^*\nabla_{\tilde Y_k}h_k(\tilde X_k^{(l+1)},\tilde Y_k^{(l)})\|_F^2\right),
\end{array}
\end{equation}
where
$$
\begin{array}{lll}
\kappa_{kl1}&=&\min\displaystyle\left\{\frac{\alpha_{kl}}{L_{k1}^{(l)}}\left(\frac{1}{(1+4\rho_1)\sigma^2_{\tilde Y_kl1}}-\frac{\alpha_{kl}}{2\sigma^4_{\tilde Y_klt}}\right),\frac{\beta_{kl}}{L_{k2}^{(l)}}\left(\frac{1}{(1+4\rho_1)\sigma^2_{\tilde X_kl1}}-\frac{\beta_{kl}}{2\sigma_{\tilde X_kls}^4}\right)\right\}.
\end{array}$$
\end{proposition}
\begin{proof}
By Lemma \ref{lipsh}, we have
$$
\begin{array}{l}
h_k(\tilde X_k^{(l)},\tilde Y_k^{(l)})-h_k(\tilde X_k^{(l+1)},\tilde Y_k^{(l)})\\
\geq -\langle \nabla_{\tilde X_k}h_k(\tilde X_k^{(l)},\tilde Y_k^{(l)}), \tilde X_k^{(l+1)}-\tilde X_k^{(l)}\rangle-\frac{L_{k1}^{(l)}}{2}\|\tilde X_k^{(l+1)}-\tilde X_k^{(l)}\|_F^2.
\end{array}
$$
Consequently, by (\ref{XXYY1}) and the first express in (\ref{XXYYZZ}), we obtain
$$
\begin{array}{l}
h_k(\tilde X_k^{(l)},\tilde Y_k^{(l)})-h_k(\tilde X_k^{(l+1)},\tilde Y_k^{(l)})\geq \displaystyle \kappa_{X_kl}\|\nabla_{\tilde X_k}h_k(\tilde X_k^{(l)},\tilde Y_k^{(l)})U_{Y_k1}^{(l)}\|_F^2,
\end{array}
$$
where $\kappa_{X_kl}=\frac{\alpha_{kl}}{L_{k1}^{(l)}}\left(\frac{1}{(1+4\rho_1)\sigma^2_{\tilde Y_kl1}}-\frac{\alpha_{kl}}{2\sigma^4_{\tilde Y_klt}}\right)$.
Similarly, we have
$$
\begin{array}{l}
h_k(\tilde X_k^{(l+1)},\tilde Y_k^{(l)})-h_k(\tilde X_k^{(l+1)},\tilde Y_k^{(l+1)})\geq \displaystyle\kappa_{Y_kl}\|(V_{\tilde X_k1}^{(l)})^*\nabla_{\tilde Y_k}h_k(\tilde X_k^{(l+1)},\tilde Y_k^{(l)})\|_F^2,
\end{array}
$$
where $\kappa_{Y_kl}=\frac{\beta_{kl}}{L_{k2}^{(l)}}\left(\frac{1}{(1+4\rho_1)\sigma^2_{\tilde X_kl1}}-\frac{\beta_{kl}}{2\sigma_{\tilde X_kls}^4}\right)$.
By summing the two inequalities above, we obtain the desired result and complete the proof.
\end{proof}

\begin{proposition}\label{Prop5}
Let $\{\tilde Q_k^{(l)}:=(\tilde X_k^{(l)},\tilde Y_k^{(l)},\tilde W_k^{(l)})\}$ be a sequence generated by Algorithm 1. Suppose that for every $k\in [m_3]$, the corresponding level set
$${\rm Lev}(f_k,\tilde Q_k^{(0)}):=\left\{\tilde Q_k\in \mathbb{C}^{m_1\times r}\times \mathbb{C}^{r\times m_2}\times \mathbb{C}^{m_1\times m_2}~|~f_k(\tilde Q_k)\leq f_k(\tilde Q_k^{(0)})\right\}$$ is bounded, where
$$
f_k(\tilde X_k,\tilde Y_k,\tilde W_k)=\frac{1}{2}\big(\|\tilde X_k\tilde Y_k-\tilde W_k\|_F^2+\rho_1\|H\tilde X_k\tilde Y_k\|^2_F+\rho_2\|\tilde W_k K^\top\|^2_F+\mu\|\tilde W_k\|^2_F\big).
$$
Then, for $k\in [m_3]$, it holds that
\begin{equation}\label{ineqZZl}
\begin{array}{l}
h_k(\tilde X_k^{(l)},\tilde Y_k^{(l)})-h_k(\tilde X_k^{(l+1)},\tilde Y_k^{(l+1)})\\
\geq (\kappa_{kl1}/\kappa_{kl2})\left(\|\nabla_{\tilde X_k}h_k(\tilde X_k^{(l)},\tilde Y_k^{(l)})U_{Y_k1}^{(l)}\|_F^2+\|(V_{\tilde X_k1}^{(l)})^*\nabla_{\tilde Y_k}h_k(\tilde X_k^{(l)},\tilde Y_k^{(l)})\|_F^2\right),
\end{array}
\end{equation}
for $l=0,1,\ldots$, where $\kappa_{kl1}$ is defined in Proposition \ref{Prop4}, and
$$
\kappa_{kl2}={\max}\left\{2,1+\frac{2L^2\alpha^2_{kl}}{(L^{(l)}_{k1})^2\sigma^4_{\tilde Y_klt}}\right\},
$$
with $L=\max \{\max \{L_{k1}(\tilde Y), L_{k2}(\tilde X)\}~|~(\tilde X,\tilde Y,\tilde W)\in {\rm Lev}(f_k,\tilde Q_k^{(0)})\}$.
\end{proposition}

\begin{proof}
It is obvious that
$$
\|\nabla h_k(\tilde X_k^{(l)},\tilde Y_k^{(l)})-\nabla h_k(\tilde X_k^{(l+1)},\tilde Y_k^{(l)})\|_F^2\leq L^2\|\tilde X_k^{(l+1)}-\tilde X_k^{(l)}\|_F^2,
$$
which implies, together with the first inequality in (\ref{XXYYZZ}),  that
\begin{equation}\label{DDXXf}
\|\nabla h_k(\tilde X_k^{(l)},\tilde Y_k^{(l)})-\nabla h_k(\tilde X_k^{(l+1)},\tilde Y_k^{(l)})\|_F^2\leq \displaystyle\frac{L^2\alpha^2_{kl}}{(L^{(l)}_{k1})^2\sigma^4_{\tilde Y_klt}}\|\nabla_{\tilde X_k}h_k(\tilde X_k^{(l)},\tilde Y_k^{(l)})U_{Y_k1}^{(l)}\|_F^2.
\end{equation}
Consequently, by using a similar way to that in \cite{BT13}, we have
$$
\begin{array}{l}
\|(V_{\tilde X_k1}^{(l)})^*\nabla_{\tilde Y_k}h_k(\tilde X_k^{(l)},\tilde Y_k^{(l)})\|_F^2\\
~~~\leq\left(\|(V_{\tilde X_k1}^{(l)})^*\big(\nabla_{\tilde Y_k}h_k(\tilde X_k^{(l)},\tilde Y_k^{(l)})-\nabla_{\tilde Y_k}h_k(\tilde X_k^{(l+1)},\tilde Y_k^{(l)})\big)\|\right.\\
~~~~~~~+\left.\|(V_{\tilde X_k1}^{(l)})^*\nabla_{\tilde Y_k}h_k(\tilde X_k^{(l+1)},\tilde Y_k^{(l)})\|\right)^2\\
~~~\leq 2\left(\|\nabla h_k(\tilde X_k^{(l)},\tilde Y_k^{(l)})-\nabla h_k(\tilde X_k^{(l+1)},\tilde Y_k^{(l)})\|_F^2\right.\\
~~~~~~~+\left.\|(V_{\tilde X1}^{(l)})^*\nabla_{\tilde Y_k}h_k(\tilde X_k^{(l+1)},\tilde Y_k^{(l)})\|_F^2\right)\\
~~~\leq\displaystyle\frac{2L^2\alpha^2_{kl}}{(L^{(l)}_{k1})^2\sigma^4_{\tilde Y_klt}}\|\nabla_{\tilde X_k}h_k(\tilde X_k^{(l)},\tilde Y_k^{(l)})U_{Y_k1}^{(l)}\|_F^2+2\|(V_{\tilde X1}^{(l)})^*\nabla_{\tilde Y_k}h_k(\tilde X_k^{(l+1)},\tilde Y_k^{(l)})\|_F^2,
\end{array}
$$
where the last inequality is due to (\ref{DDXXf}). Therefore, we further obtain
$$
\begin{array}{l}
\|\nabla_{\tilde X_k}h_k(\tilde X_k^{(l)},\tilde Y_k^{(l)})U_{Y_k1}^{(l)}\|_F^2+\|(V_{\tilde X_k1}^{(l)})^*\nabla_{\tilde Y_k}h_k(\tilde X_k^{(l)},\tilde Y_k^{(l)})\|_F^2\\
\leq \kappa_{kl2}\left(\|\nabla_{\tilde X_k}h_k(\tilde X_k^{(l)},\tilde Y_k^{(l)})U_{Y_k1}^{(l)}\|_F^2+\|(V_{\tilde X1}^{(l)})^*\nabla_{\tilde Y_k}h_k(\tilde X_k^{(l+1)},\tilde Y_k^{(l)})\|_F^2\right),
\end{array}
$$
which implies, together with Proposition \ref{Prop4},  that we have obtained (\ref{ineqZZl}) and complete the proof.
\end{proof}

\begin{theorem}\label{ThKKT}
Let $\{(\tilde X_k^{(l)},\tilde X_k^{(l)},\tilde W_k^{(l)}\}$ be a sequence generated by TCTF2R Algorithm. Suppose that, for $k\in [m_3]$,  there exits real number $c_k>0$ such that $\kappa_{kl1}/\kappa_{kl2}\geq c_k$ for any $l=0,1,2,\ldots$. Suppose that ${\rm Lev}(f_k,\tilde Q_k^{(0)})$ is bounded for every $k\in [m_3]$. For every convergence subsequence $\{(\tilde X_k^{(l_j)}\tilde Y_k^{(l_j)},\tilde W_k^{(l_j)})\}_{j=1}^\infty$ of $\{(\tilde X_k^{(l)}\tilde Y_k^{(l)},\tilde W_k^{(l)})\}_{l=1}^\infty$ with its limit being $(\hat {W}_{k\star},\tilde{W}_{k\star})$, if the matrix $\hat {W}_{k\star}$ can be decomposed into $\hat {W}_{k\star}=\tilde X_{k\star}\tilde Y_{k\star}$ satisfying $ \tilde X_k^{(l_j)}\rightarrow\tilde X_{k\star},~ \tilde Y_k^{(l_j)}\rightarrow\tilde Y_{k\star}$ as $j\rightarrow \infty$, then it holds that
\begin{equation}\label{XYkKKTstar}
\left\{
\begin{array}{l}
\nabla_{\tilde X}h_k(\tilde X_{k\star},\tilde Y_{k\star},\tilde W_{k\star})U_{\tilde Y_{k\star}1}=0,\\
(V_{\tilde X_{k\star}1})^*\nabla_{\tilde Y_k}h_k(\tilde X_{k\star},\tilde Y_{k\star},\tilde W_{k\star})=0
\end{array}
\right.
\end{equation}
for every $k\in [m_3]$, and
\begin{equation}\label{WKKTcon}
\left\{\begin{array}{ll}
(W_{k\star}K_{\mu \rho_2}-\mathcal{X}_\star*\mathcal{Y}_\star(:,:,k))_{ij}=0,~~{\rm for~}k\in [m_3],~(i,j)\in \Omega^C_k,\\
P_{\Omega}(\mathcal{W}_\star)=P_{\Omega}(\mathcal{G}).
\end{array}
\right.
\end{equation}
where $W_{k\star}$ is the $k$-th frontal slice of $\mathcal{W}_{\star}$ and is obtained by (\ref{barAAk}) from $\tilde W_{l\star} (l\in [m_3])$, and
$\mathcal{X}_\star\in \mathbb{R}^{m_1\times r}$ is a tensor whose $k$-th frontal slice is the matrix $X_{k\star}~(k\in [m_3])$ which is obtained from $\tilde X_{k\star}~(k\in [m_3])$ through the formula (\ref{barAAk}), and $\mathcal{Y}_\star\in \mathbb{R}^{r\times m_2}$ is similar.
\end{theorem}
\begin{proof}
By Theorem \ref{GlobConvergM}, we have $\{f_0(\mathcal{X}^{(l)},\mathcal{Y}^{(l)},\mathcal{W}^{(l)})\}$  monotonically converges to $\tilde f_{\star}:=f(\mathcal{X}_{\star},\mathcal{Y}_{\star},\mathcal{W}_{\star})$ on the whole sequence. Therefore, it holds that $$f_0(\mathcal{X}^{(l)},\mathcal{Y}^{(l)},\mathcal{W}^{(l)})-f_0(\mathcal{X}^{(l+1)},\mathcal{Y}^{(l+1)},\mathcal{W}^{(l)})\rightarrow 0$$ as $l\rightarrow\infty$. Moreover, for every $k\in [m_3]$, it holds that
\begin{equation}\label{XYlimit}
\begin{array}{l}
f_0(\mathcal{X}^{(l)},\mathcal{Y}^{(l)},\mathcal{W}^{(l)})-f_0(\mathcal{X}^{(l+1)},\mathcal{Y}^{(l+1)},\mathcal{W}^{(l)})\\
\geq \displaystyle\frac{1}{m_3}\big(h_k(\tilde X_k^{(l)},\tilde Y_k^{(l)})-h_k(\tilde X_k^{(l+1)},\tilde Y_k^{(l+1)})\big).
\end{array}
\end{equation}
 Then, by letting $j\rightarrow \infty$ in (\ref{XYlimit}), the desired result (\ref{XYkKKTstar}) follows from Proposition \ref{Prop5}. Moreover, from the design of TCTF2R Algorithm, we see that $ Z^{(l_j)}_k\rightarrow \mathcal{X}_\star*\mathcal{Y}_\star(:,:,k)$ as $j\rightarrow \infty$, hence (\ref{WKKTcon}) follows from the KKT condition of (\ref{WETP-optim20}).
\end{proof}

\begin{remark}
Since model (\ref{TETP-optim}), and hence (\ref{kXYETP-optim20}) is nonconvex, we are only able to establish convergence to a stationary point under a suitable assumption. In addition, it is obvious that ${\rm rank}(\hat W_{k\star})\leq r$ for $k\in [m_3]$. It is worth noting that, due to the non-uniqueness of the singular value decomposition of the matrix, it is usually difficult to find specific $\tilde X_{k\star}$ and $\tilde Y_{k\star}$ satisfying  $\hat W_{k\star}=\tilde X_{k\star}\tilde Y_{k\star}$ and $ \tilde X_k^{(l_j)}\rightarrow\tilde X_{k\star},~ \tilde Y_k^{(l_j)}\rightarrow\tilde Y_{k\star}$ as $j\rightarrow \infty$. However, from the proof of Theorem \ref{ThKKT}, we see that, for given arbitrarily small $\varepsilon>0$, under the given condition, there exists a positive integer $\bar j$ such that
\begin{equation}\label{AKKTstar}
\left\{
\begin{array}{l}
\|\nabla_{\tilde X_k}h_k(\tilde X_{k\star},\tilde Y_{k\star},\tilde W_{k\star})U_{\tilde Y_{k\star}1}\|_F\leq \varepsilon,\\
\|(V_{\tilde X_{k\star}1})^*\nabla_{\tilde Y_k}h_k(\tilde X_{k\star},\tilde Y_{k\star},\tilde W_{k\star})\|_F\leq \varepsilon,\\
(W_{k\star}K_{\mu\rho_2}-\mathcal{X}_\star*\mathcal{Y}_\star(:,:,k))_{ij}=0,~~\forall~(i,j)\in \Omega^C_k,\\
P_{\Omega}(\mathcal{W}_\star)=P_{\Omega}(\mathcal{G})
\end{array}
\right.
\end{equation}
for every $k\in [m_3]$ and $j\geq \bar j$, which is regarded as an approximation KKT systems of (\ref{kXYETP-optim20}) and (\ref{WETP-optim20}). In particular, when ${\rm rank }(\hat W_{k\star})=r$, by the well-known perturbation analysis theory of matrix singular-values, we know that, the $(\tilde{X}_k^{(l)},\tilde{Y}_k^{(l)})$ mentioned above satisfies ${\rm rank}(\tilde{X}_k^{(l)})={\rm rank}(\tilde{Y}_k^{(l)})=r$ for $l\geq \bar l$, whenever $\tilde{X}_k^{(l)}\tilde{Y}_k^{(l)}\rightarrow \hat W_{k\star}$, which implies $U_{\tilde Y_k1}^{(l)}=U_{\tilde Y_k}^{(l)}$ and $V_{\tilde X_k1}^{l}=V_{\tilde X_k}^{(l)}$ are unitary matrices for $l\geq \bar l$. In this case, the considered systems (\ref{XYkKKTstar}) and (\ref{AKKTstar})  reduce the corresponding classic (approximation) KKT systems of (\ref{kXYETP-optim20}), respectively.
\end{remark}

\section{Experiments}\label{Experiments}
In this section, numerical experiments are presented to verify the performance of our algorithm. We use the normalized mean absolute error (NMAE) in the missing values as a metric of the recovered data.
The NMAE is defined as follows
$$
\mbox{NMAE}=\frac{\sum_{(i,j)\notin \Omega} |X_{ij}-\hat{X}_{ij}|}{\sum_{(i,j)\notin \Omega} |X_{ij}|},
$$
where $X$ and $\hat{X}$ are original data and estimated data, respectively. We evaluate the performance on random missing pattern, i.e., each entry of the given traffic matrix (TM) data is uniformly and randomly missing, and reconstructing the TM 10 times. The results presented show the mean NMAE.

\begin{figure}[h]
\centering
  \includegraphics[width = 1.0\textwidth]{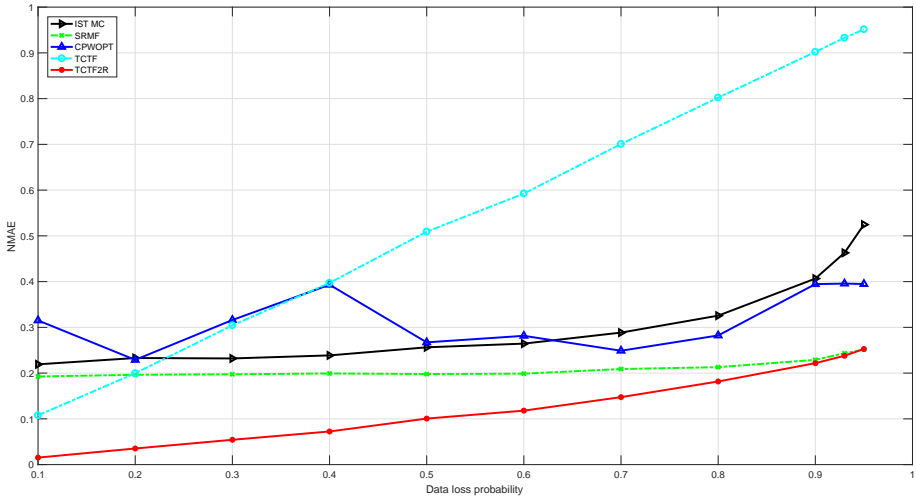}
  \caption{NMAE under random missing from Abilene Data}
  \label{A.NMAE}
\end{figure}

\begin{figure}[h]
\centering
  \includegraphics[width = 1.0\textwidth]{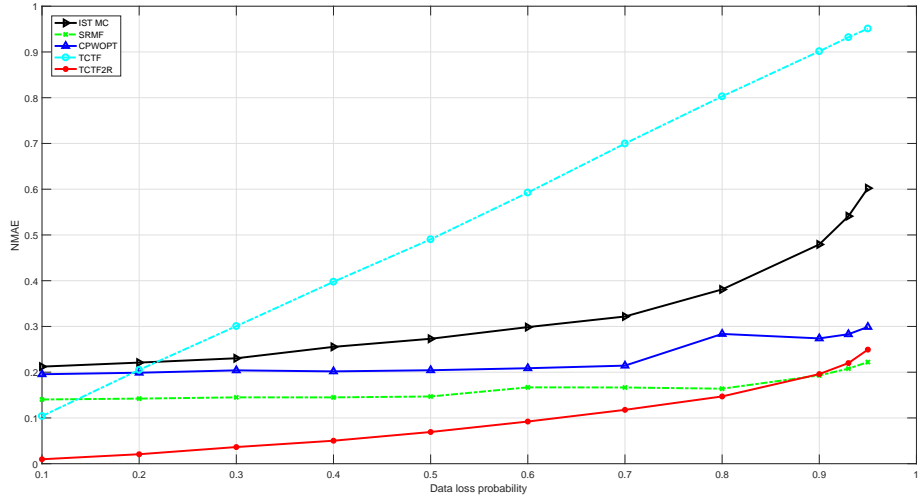}
  \caption{NMAE under random missing from G\'{E}ANT Data}
  \label{G.NMAE}
\end{figure}

To verify the performance of the proposed method, we compare the NMAE of several data completion methods.
The first is Matrix completion (IST MC) by the nuclear norm minimization \cite{CR09,Majumdar20}, which does not consider the spatio-temporal structure of the internet traffic matrix. The second is Sparsity regularized matrix factorization (SRMF) method \cite{RZWQ12,ZRWQ09}, which is a low-rank matrix completion approach with a spatio-temporal regularization. The third approach is CP-WOPT, which is a CP tensor completion method with a spatio-temporal regularization \cite{ADKM11,ZZXC15}. The last one is the normal TCTF method \cite{ZLLZ18}.
In all experiments, we set $\mbox{tol=1e-6}$. The platform is Matlab 2016b under Windows 10 on a PC of a 2.71GHz CPU and 32GB memory.

\begin{figure}
$$
\begin{array}{c}
\includegraphics[width=1.0\textwidth]{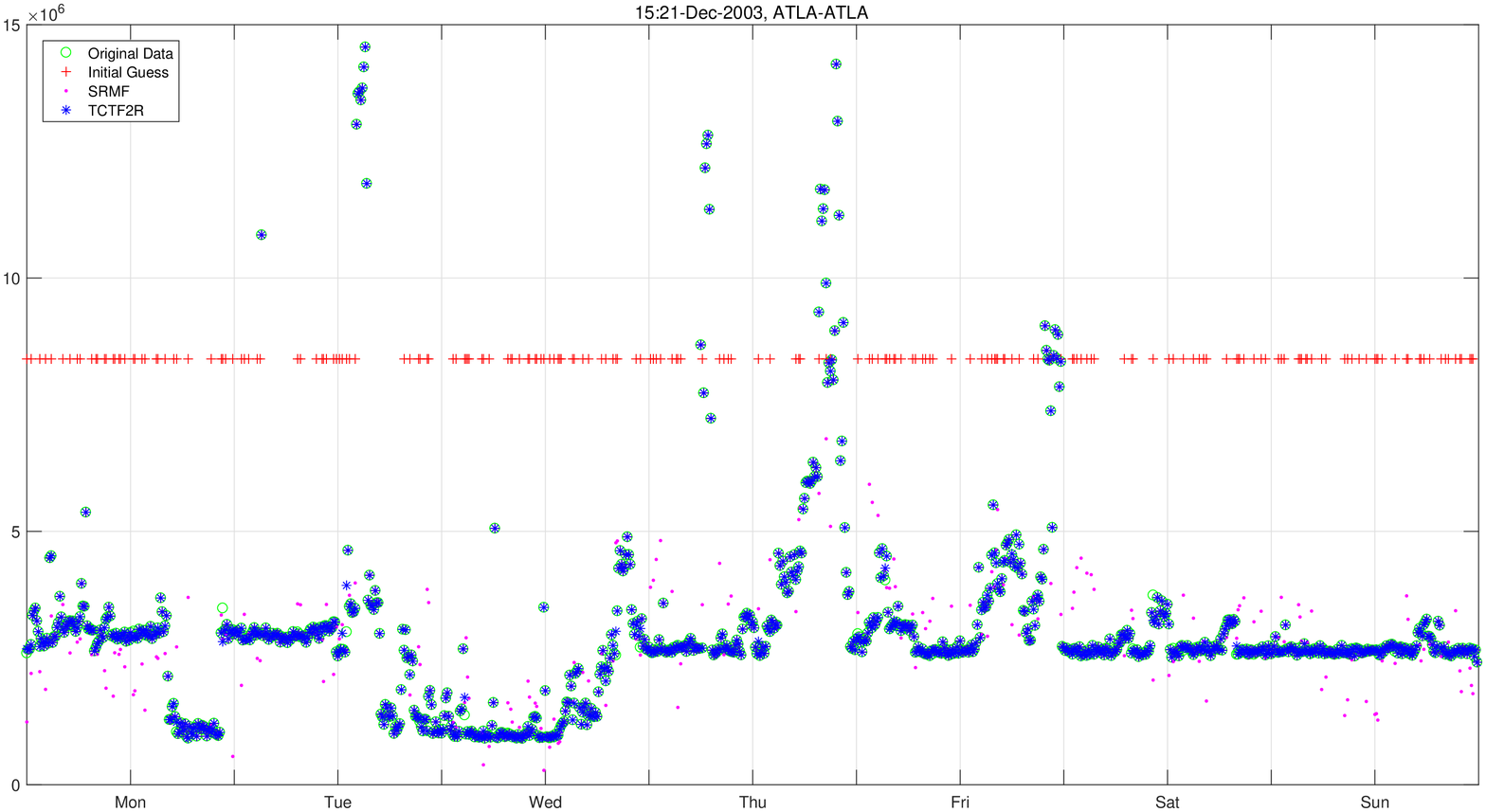}\\
(i)\mbox{The data loss probability is 20\%} \\
\includegraphics[width=1.0\textwidth]{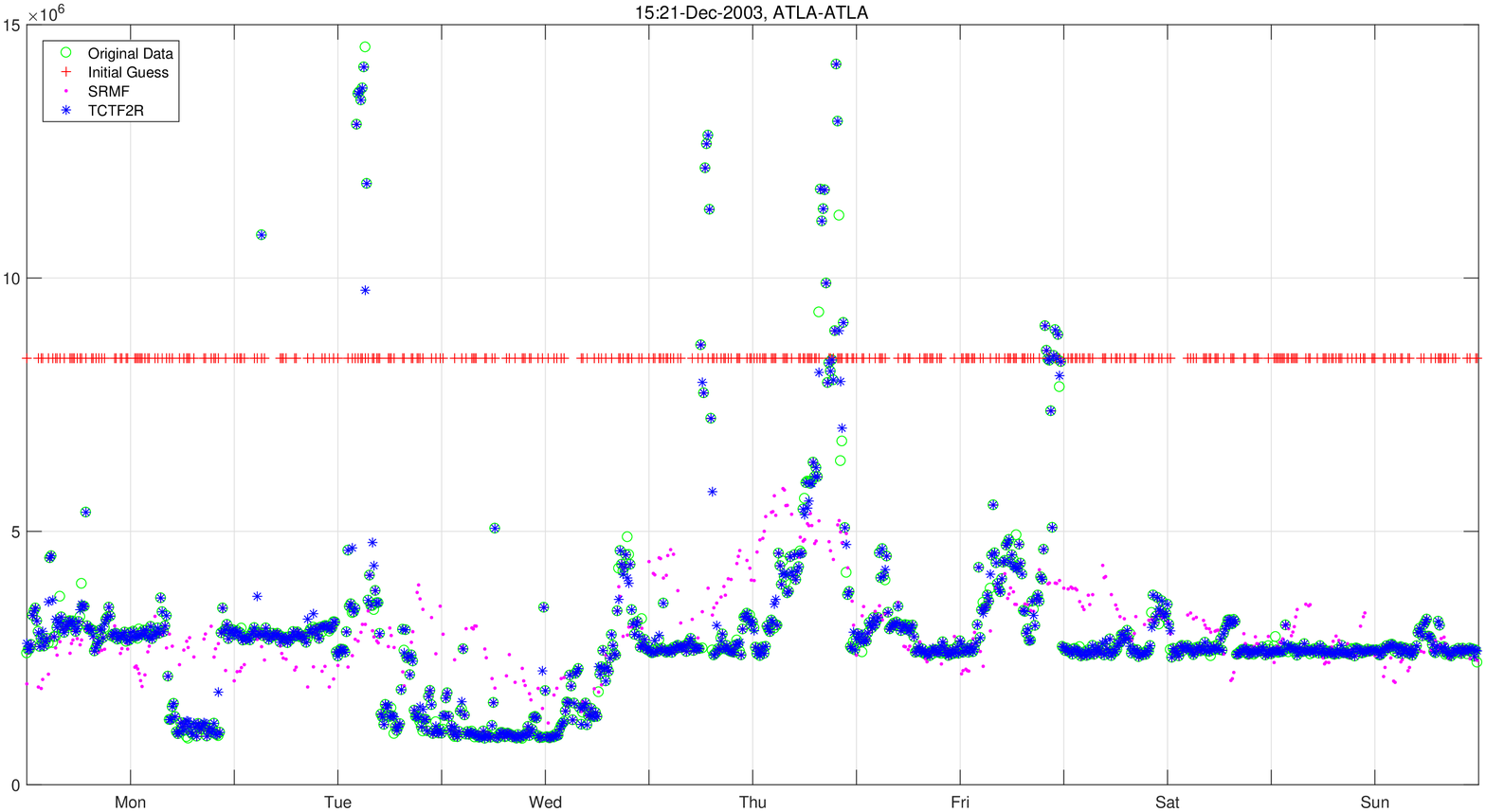}\\
(ii) \mbox{The data loss probability is 40\%}
\end{array}
$$
\caption{Recovery of the ATLA-ATLA OD pair: SRMF v.s. TCTF2R}.
\end{figure}

\begin{figure}
$$
\begin{array}{c}
\includegraphics[width=1.0\textwidth]{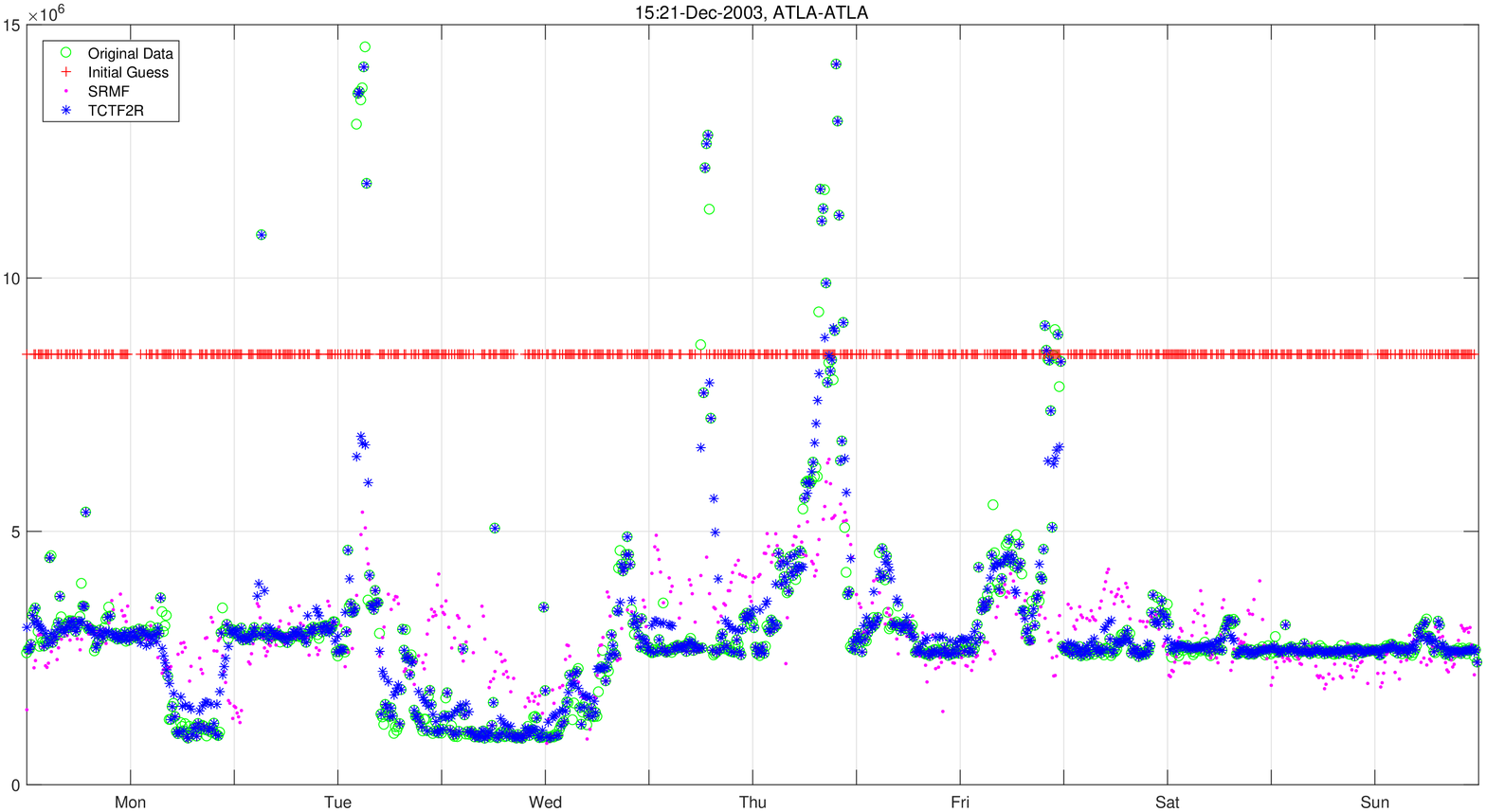}\\
(iii)\mbox{The data loss probability is 60\%}\\
\includegraphics[width=1.0\textwidth]{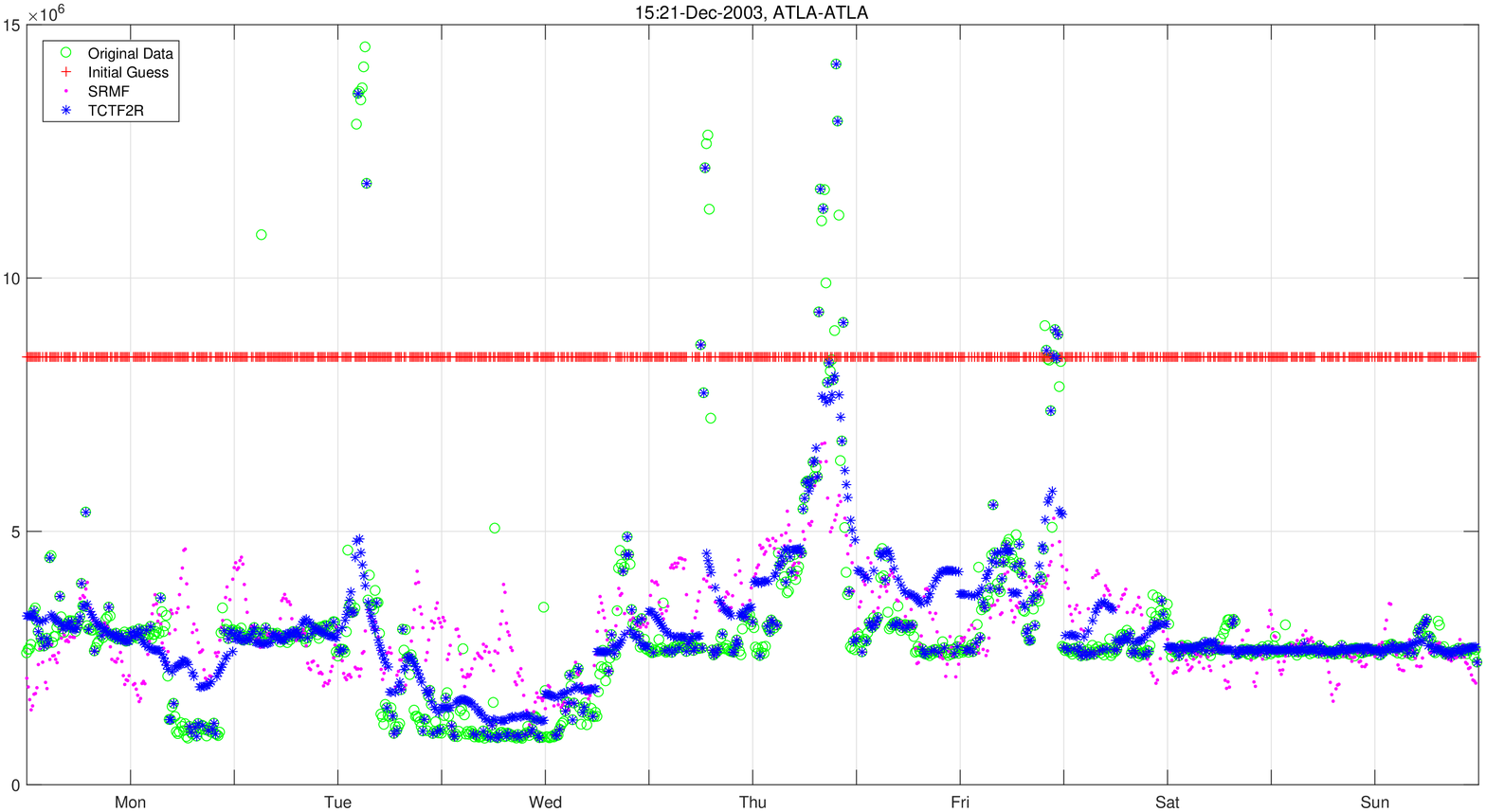}\\
(iv) \mbox{The data loss probability is 80\%}
\end{array}
$$
\caption{Recovery of the ATLA-ATLA OD pair: SRMF v.s. TCTF2R}.
\end{figure}

Our experiments are tested on two real-world traffic dataset. First, we consider the widely used Abilene dataset \cite{Abilene04}, in which there are 11 routers and hence 121 OD pairs. For each OD pair, a count of network traffic flow is recorded for every 10 minutes in a week from Dec 15, 2003 to Dec 21, 2003. Thus, there are $7\times 24\times 60/10 = 1008$ numbers for each OD pair. In this way, we obtain an internet traffic matrix $X$ with size 121-by-1008. The second real-world dataset is the G\'{E}ANT traffic dataset \cite{UQLB06} in which there are 23 routers and hence 529 OD pairs. For each OD pair, a count of network traffic flow is recorded for every 15 minutes in a day. Hence there are $24\times60/15 = 96$ data in a day for an OD pair. We choose one week traffic data collected from March 26, 2005 to April 2, 2005. In this way, we obtain an internet traffic matrix $X$ with size 529-by-672.

 Firstly, we investigate the performance of the tested approaches on the Abilene data with the data loss probability ranging from 10\% to 95\%. We compute the recovered internet traffic data and calculate the corresponding NMAEs, which are illustrated in Figure 1. As we can see, the proposed method (TCTF2R) significantly outperforms the normal TCTF method under variety of missing rates. Similarly, spatio-temporal regularized SRMF method performs better than non-regularized IST MC method. This phenomenon implies
that the spatio-temporal structure in the internet traffic matrix are valuable and has
been used to improve the recovery accuracy. Further, when the data loss probability is less than 90\%, the
proposed TCTF2R performs better than all other approaches. When 90\% data have been lost, the normalized mean absolute error (NMAE) of TCTF2R method is about 0.22, which is slightly better than that of SRMF method. SRMF is just behind TCTF2R, and achieves robust performance over the whole loss range. Further, we apply these approaches on the G\'{E}ANT traffic dataset, and similar results could be observed from Figure 2.

In order to check the real recovery accuracy, we illustrate the recovered data for the ATLA-ATLA OD pair of Abilene data, which is the tubal vector $(1,1,:)$ of the internet traffic tensor with size $11\times11\times1008$. As shown in Figures 3-4, TCTF2R method can recover the missing data in high accuracy, which is much better than SRMF method. When the data lost probability is less than 40\%, TCTF2R method can recover the data in very high accuracy. But for the SRMF method, there are lots of points outside the green circle, which means that SRMF method can not accurately recover the original data. As we can see from the Fig. 4, with the increase of data loss rate, the accuracy of recovery is further reduced. However, compared with SRMF method, the TCTF2R method still has better performance. According to these fact, we illustrate the significance of the proposed tensor completion method for the recovery of internet traffic data.

\section{Conclusion}\label{Conclusion}
The internet data recovery model has its own special structure. How to decompose high-order tensors according to the detachable structure of the model has become the key issue of designing efficient algorithms for internet network traffic data recovery. In this paper, t-product and t-SVD of tensors and the closely related tensor tubal rank were used to establish a low-rank optimization model for internet traffic data recovery. Furthermore, using rapid discrete Fourier transform, the established model was transformed into a structured model with separability, and the temporal stability and periodicity features in internet traffic data were fully reflected in this model. Based upon this, with the help of the generalized inverse matrix, an easy-to-operate and relatively effective algorithm was proposed. Some convergence properties of the proposed algorithm was analysed. The numerical simulations on widely used real-world internet datasets, e.g., Abilene dataset and G\'{E}ANT traffic dataset, showed the excellent performance of the proposed method, especially when the data loss rate is less than 80\%, the recovery effect is more outstanding, and the relative error rate decreases rapidly as the data loss rate decreases.
Due to the limitation of our computing platform, the algorithm proposed in this paper is not implemented in parallel. However, the detachable separation characteristics of the model in this paper provided the possibility of designing more efficient parallel algorithms for solving internet network traffic data recovery in the future.

\section*{Acknowledgment}
C. Ling and G. H. Yu were supported in part by National Natural Science Foundation of China (Nos. 11971138 and 11661007) and Natural Science Foundation of Zhejiang Province (Nos. LY19A010019 and LD19A010002).


\begin{thebibliography}{99}

\bibitem{Abilene04} The Abilene Observatory Data Collections. Accessed: May 2004.
[Online]. Available: http://abilene.internet2.edu/observatory/data-collections.html.

\bibitem{ADKM11} E. Acar, D. M. Dunlavy, T. G. Kolda and M. M\o rup, ``Scalable tensor factorizations for incomplete data'', {\sl Chemometrics and Intelligent Laboratory Systems \bf 106} (2011), 41-56.

\bibitem {ACRUW06} D. Alderson, H. Chang, M. Roughan, S. Uhlig and W. Willinger, ``The many facets of internet topology and traffic'', {\sl   Networks and Heterogeneous Media \bf 1(4)} (2006), 569-600.

\bibitem{AMDJ16} M. T. Asif, N. Mitrovic, J. Dauwels and P. Jaillet, ``Matrix and tensor based methods for missing data estimation in large traffic networks'', {\sl IEEE Transactions on Intelligent Transportation Systems \bf 17} (2016), 1816-1825.

\bibitem{BT13} A. Beck and L. Tetruashvili, ``On the convergence of block corrdinate descent type methods'', {\sl SIAM Journal on Optimization \bf 23} (2013), 2037-2060.

\bibitem{Ber08} Dennis. S. Bernstein, ``Matrix Mathematics: theory, facts, and formulas'', 2th edition, Princeton University Press, Princeton and Oxford, 2008.

\bibitem {Bert99} D. P. Bertsekas, ``Nonlinear Programming'', 2th edition, Athena Scientific, Belmont, MA, 1999.


\bibitem{B83} D. H. Brandwood, ``A complex gradient operator and its application in adaptive array theory'', {\sl IEE Proceedings H: Microwaves, Optics and Antennas \bf 130(1)} (1983), 11-16.

\bibitem{CCS10} J. F. Cai, E. J. Cand\`{e}s and Z. Shen, ``A singular value thresholding algorithm for matrix completion'', {\sl SIAM Journal on Optimization \bf 20(4)} (2010), 1956-1982.

\bibitem{CR09} E. J. Cand\`{e}s and B. Recht, ``Exact matrix completion via convex optimization'', {\sl Foundations of Computational Mathematics \bf 9(6)} (2009), 717-772.

\bibitem{CC70} J. D. Carroll and J. J. Chang, ``Analysis of individual differences in multidimensional scaling via an $n$-way generalization of `Eckart-Young' decomposition'', {\sl Psychometrika \bf 35 (3)} (1970), 283-319.

\bibitem {Cah98} R. S. Cahn, ``Wide Area Network Design'', San Mateo, CA: Morgan Kaufman, 1998.


\bibitem{CSZZC19} B. Chen, T. Sun, Z. Zhou, Y. Zeng, L. Cao, ``Nonnegative tensor completion via low-rank Tucker decomposition: model and algorithm'', {\sl IEEE Access \bf 7} (2019) 95903-95914.


\bibitem{CQZXH14} Y. C. Chen, L. Qiu, Y. Zhang, G. Xue and Z. Hu, ``Robust network compressive sensing'', {\sl Proceedings of the 20th Annual International Conference on Mobile Computing and Networking} (2014), 545-556.

\bibitem {DH15} C. Da Silva and F. J. Herrmann, ``Optimization on the hierarchical Tucker manifold-applications to tensor completion'', {\sl Linear Algebra and its Applications \bf481} (2015), 131-173.

\bibitem{Do06} D. L. Donoho, ``Compressed sensing'', {\sl IEEE Transactions on Information Theory \bf52} (2006), 1289-1306.

\bibitem{DCYG13} R. Du, C. Chen, B. Yang and X. Guan, ``VANET based traffic estimation: A matrix completion approach'', {\sl IEEE Global Communications Conference} (2013), 30-35.


\bibitem{GRY11} S. Gandy, B. Recht and I. Yamada, ``Tensor completion and low-$n$-rank tensor recovery via convex optimization'', {\sl Inverse Problems \bf 27(2)} (2011), 025010.

\bibitem{GV13} G. H. Golub and C. F. Van Loan, ``Matrix Computations'', 4th edition, Johns Hopkins University Press, Baltimore, 2013.

\bibitem{GC12} G. G\"{u}rsun and M. Crovella, ``On traffic matrix completion in the internet'', {\sl Proceedings of the ACM SIGCOMM Internet Measurement Conference} (2012), 399-412.


\bibitem{H70} R. A. Harshman, ``Foundations of the PARAFAC procedure: Models and conditions for an `explanatory' multi-modal factor analysis'', {\sl  UCLA Work Papers Phonetics \bf 16 (1)} (1970), 1-84.

\bibitem{HJ94} Roger A. Horn, Charles R. Johnson, ``Topics in Matrix Analysis'', Cambridge University Press, 1994.

\bibitem{KM11} M. E. Kilmer and C. D. Martin, ``Factorization strategies for third-order tensors'', {\sl Linear Algebra and its Applications \bf 435} (2011), 641-658



\bibitem{KBHH13} M. E. Kilmer, K. Braman, N. Hao, and R. C. Hoover. ``Third-order tensors as operators on matrices: a theoretical and computational framework with applications in imaging'', {\sl SIAM Journal on Matrix Analysis and Applications \bf 34} (2013), 148-172.

\bibitem{KB09} T. G. Kolda and B. W. Bader, ``Tensor decompositions and applications'', {\sl SIAM Review. \bf 51} (2009), 455-500.

\bibitem{LPCDKT03} A. Lakhina, K. Papagiannaki, M. Crovella, C. Diot, E. D. Kolaczyk and N. Taft, ``Structural analysis of network traffic flows'', {\sl ACM SIGMETRICS Performance Evaluation Review}  (2004), 61-72.


\bibitem{LDGL13} Y. Liao, W. Du, P. Geurts and G. Leduc, ``DMFSGD: A decentralized matrix factorization algorithm for network distance prediction'', {\sl IEEE/ACM Transactions on Networking \bf 21 (5)} (2013) 1511-1524

\bibitem{Majumdar20} A. Majumdar, ``Matrix Completion via Thresholding", MATLAB Central File Exchange. Retrieved April 2, 2020.


\bibitem{MG13} M. Mardani and G. B. Giannakis, ``Robust network traffic estimation via sparsity and low rank'', {\sl IEEE International Conference on Acoustics, Speech and Signal Processing} (2013), 4529-4533.

\bibitem {M20} E. Moore, ``On the reciprocal of the general algebraic matrix'', {\sl Bulletin of the American Mathematical Society \bf26} (1920), 394-395.

\bibitem{NJG13} L. Nie, D. Jiang, and L. Guo, ``A power laws-based reconstruction approach to end-to-end network traffic'', {\sl Journal of Network and Computer Applications \bf 36(2)} (2013) 898-907.

\bibitem{P55} R. Penrose, ``A generalized inverse for matrices'', {\sl Mathematical Proceedings of the Cambridge Philosophical Society \bf 51} (1955) 406-413.

\bibitem{RTZ03} M. Roughan, M. Thorup and Y. Zhang, ``Traffic engineering with estimated traffic matrices'', {\sl Proc. ACM IMC 2003}, 248-258.

\bibitem{RZWQ12} M. Roughan, Y. Zhang, W. Willinger and L. Qiu, ``Spatio-temporal compressive sensing and Internet traffic matrices (extended version)'', {\sl IEEE/ACM Transactions on Networking \bf 20} (2012), 662-676.

\bibitem{SLH19} K. Shang, Y. F. Li and Z. H. Huang, ``Iterative $p$-shrinkage thresholding algorithm for low Tucker rank tensor recovery'', {\sl Information Sciences \bf482} (2019), 374-391.



\bibitem{TWSJR16} H. Tan, Y. Wu, B. Shen, P. J. Jin and B. Ran, ``Short-term traffic prediction method based on dynamic tensor completion'', {\sl IEEE Transactions on Intelligent Transportation Systems \bf 17} (2016), 2123-2133.

\bibitem{TYFWR13} H. Tan, Z. Yang, G. Feng, W. Wang and B. Ran, ``Correlation analysis for tensor-based traffic data imputation method'', {\sl Procedia-Social and Behavioral Sciences \bf 96} (2013), 2611-2620.


\bibitem{TH77} R. C. Thompson, ``Inertial properties of eigenvalues II'', {\sl Journal of Mathematical Analysis and Applications \bf 58} (1977), 572-577.

\bibitem{Tucker66} L. Tucker, ``Some mathematical notes on three-mode factor analysis'', {\sl Psychometrika \bf 31(3)} (1966), 279-311.

\bibitem{TR13} P. Tune and M. Roughan, ``Internet traffic matrices: A Primer'', in: H. Haddadi and O. Bonaventure, eds., {\sl Recent Advances in Networking \bf 1} (2013), 1-56.

\bibitem {TR15} P. Tune and M. Roughan, ``Spatiotemporal traffic matrix synthesis'', {\sl Proceedings of the 2015 ACM Conference on
Special Interest Group on Data Communication} (2015), 579-592.

\bibitem{UQLB06} S. Uhlig, B. Quoitin, J. Lepropre and S. Balon, ``Providing public intradomain traffic matrices to the research community'', {\sl ACM SIGCOMM Computer Communication Review \bf 36} (2006), 83-86.



\bibitem{V96} Y. Vardi, ``Network tomography: Estimating source-destination traffic intensities from link data'', {\sl Journal of the American Statistical Association, \bf 91(433)} (1996), 365-377.


\bibitem{XLWXWZ18} K. Xie, X. Li, X. Wang, G. Xie, J. Wen and D. Zhang, ``Graph based tensor recovery for accurate Internet anomaly detection'', {\sl IEEE INFOCOM 2018-The 37th Annual IEEE International Conference on Computer Communications} 2018, 1502-1510.


\bibitem{XPWXWCZQ18} K. Xie, C. Peng, X. Wang, G. Xie, J. Wen, J. Cao, D. Zhang and Z. Qin, ``Accurate recovery of internet traffic data under variable rate measurements'', {\sl IEEE/ACM Transactions on Networking \bf 26} (2018) 1137-1150.

\bibitem{XTWXWZ19} K. Xie, J. Tian, X. Wang, G. Xie, J. Wen and D. Zhang, ``Efficiently inferring top-$k$ elephant flows based on discrete tensor completion'', {\sl IEEE INFOCOM 2019 - The 38th Annual IEEE International Conference on Computer Communications} 2019, 2170-2178.

\bibitem{XWWCXOWCZ19} K. Xie, X. Wang, X. Wang, Y. Chen, G. Xie, Y. Ouyang, J. Wen, J. Cao and D. Zhang, ``Accurate recovery of missing network measurement data with localized tensor completion'', {\sl IEEE/ACM Transactions on Networking \bf 27(6)} 2019, 2222-2235.


\bibitem{X15} K. Xie, L. Wang, X. Wang, G. Xie et al., ``Sequential and adaptive sampling for matrix completion in network monitoring systems'', {\sl Proc. IEEE INFOCOM} (2015), 2443-2451.

\bibitem{XWWXWZCZ18} K. Xie, L. Wang, X. Wang, G. Xie, J. Wen, G. Zhang, J. Cao and D. Zhang, ``Accurate recovery of internet traffic data: A sequential tensor completion approach'', {\sl IEEE/ACM Transactions on Networking \bf 26} (2018), 793-805.

 \bibitem {WYZ12} Z. Wen, W. Yin, and Y. Zhang, ``Solving a low-rank factorization model for matrix completion by a nonlinear successive over-relaxation algorithm'', {\sl Mathematical Programming Computation \bf 4(4)} (2012), 333-361.

\bibitem{YHHH16} L. Yang, Z. H. Huang, S. L. Hu and J. Y. Han, ``An iterative algorithm for third-order tensor multi-rank minimization'', {\sl Computional Optimization and  Applications \bf 63} (2016), 169-202.


\bibitem{ZRLD05} Y. Zhang, M. Roughan, C. Lund and D. L. Donoho, ``Estimating point-to-point and point-to-multipoint traffic matrices: An information theoretic approach'', {\sl IEEE/ACM Transactions on Networking \bf 13(5)} 2005, 947-960.

\bibitem{ZRWQ09} Y. Zhang, M. Roughan, W. Willinger and L. Qiu, ``Spatio-temporal compressive sensing and internet traffic matrices'', {\sl ACM SIGCOMM Computer Communication Review - SIGCOMM '09 \bf 39} (2009), 267-278.

\bibitem{ZZXC15} H. Zhou, D. Zhang, K. Xie and Y. Chen, ``Spatio-temporal tensor completion for imputing missing internet traffic data'', {\sl IEEE 34th International Performance Computing and Communications Conference (IPCCC)} (2015).

\bibitem {ZZXW14} H. Zhou, D. F. Zhang, K. Xu and X. Y. Wang, ``Data reconstruction in internet traffic matrix'', {\sl Communications, China \bf 11} (2014), 1-12.

\bibitem{ZLLZ18} P. Zhou, C. Y. Lu, Z. C. Lin and C. Zhang, ``Tensor factorization for low-rank tensor completion'', {\sl IEEE Transactions on Image Processing \bf 3} (2018), 1152-1163.

\bibitem{ZLNLZ17} R. Zhu, B. Liu, D. Niu, Z. Li and H. V. Zhao, ``Network latency estimation for personal devices: A matrix completion approach'', {\sl IEEE/ACM Transactions on Networking \bf 25(2)} (2017), 724-737.
\end{thebibliography}
\end{document}